\documentclass[12pt, a4paper,reqno]{amsart}

\usepackage[T1]{fontenc}
\usepackage[utf8]{inputenc}

\usepackage{color}
\definecolor{bleu_sombre}{rgb}{0,0,0.6}  \definecolor{rouge_sombre}{rgb}{0.8,0,0}\definecolor{vert_sombre}{rgb}{0,0.6,0}
\usepackage[plainpages=false,colorlinks,linkcolor=bleu_sombre,
citecolor=rouge_sombre,urlcolor=vert_sombre,breaklinks]{hyperref}

\usepackage[english]{babel}
\usepackage{geometry}

\usepackage{amsmath,amssymb,amsthm,graphicx,amsfonts,url,enumerate,dsfont,stmaryrd, mathrsfs, csquotes}

\usepackage{pgf,tikz,pgfplots}
\usetikzlibrary{arrows}
\definecolor{uuuuuu}{rgb}{0.26666666666666666,0.26666666666666666,0.26666666666666666}

\theoremstyle{plain}
\newtheorem{theorem}{{Theorem}}[section]
\newtheorem*{theorem*}{{Theorem}}
\newtheorem{proposition}[theorem]{Proposition}
\newtheorem{conjecture}[theorem]{Conjecture}
\newtheorem*{proposition*}{Proposition}
\newtheorem{corollary}[theorem]{Corollary}
\newtheorem*{corollary*}{Corollary}
\newtheorem{lemma}[theorem]{Lemma}
\newtheorem{assumption}[theorem]{Assumption}
\newtheorem*{lemma*}{Lemma}
\theoremstyle{definition}

\newtheorem*{definition*}{Definition}
\theoremstyle{remark}
\newtheorem{remark}[theorem]{Remark}
\newtheorem{notation}[theorem]{Notation}

\makeatletter

\@addtoreset{equation}{section}
\makeatother

\renewcommand{\leq}{\leqslant}	\renewcommand{\geq}{\geqslant}
\def\seq#1{\left<#1\right>}
\def\sep#1{\left(#1\right)}
\def\norm#1{\left\Vert#1\right\Vert} 

\newcommand{\R}{\mathbb{R}}	
\newcommand{\C}{\mathbb{C}}
\newcommand{\N}{\mathbb{N}}	
\newcommand{\Z}{\mathbb{Z}}
\newcommand{\dd}{\mathrm{d}}
\newcommand{\ooo}{\mathscr{O}}	
\newcommand{\sss}{\mathcal{S}}
\newcommand{\eps}{\varepsilon}
\newcommand{\Dom}{\mathrm{Dom}\,}
\newcommand{\Op}{\mathrm{Op}^W_\hbar}

\renewcommand{\Re}{\mathrm{Re}\,}

\begin{document}

\title[Purely magnetic tunneling effect]{Purely magnetic tunneling effect\\ in two dimensions}

\author[V. Bonnaillie-No\"el]{Virginie Bonnaillie-No\"el}
\address[V. Bonnaillie-No\"el]{D\'epartement de math\'ematiques et applications, \'Ecole normale sup\'erieure, CNRS, Universit\'e PSL, F-75005 Paris, France}
\email{bonnaillie@math.cnrs.fr}

\author[F. H\'erau]{Fr\'ed\'eric H\'erau}
\address[F. H\'erau]{LMJL - UMR6629, Universit\'e de Nantes, CNRS, 2 rue de la Houssini\`ere, BP 92208, F-44322 Nantes cedex 3, France}
\email{herau@univ-nantes.fr}

\author[N. Raymond]{Nicolas Raymond}
\address[N. Raymond]{Laboratoire Angevin de Recherche en Mathématiques, LAREMA, UMR 6093, UNIV Angers, SFR Math-STIC, 2 boulevard Lavoisier 49045 Angers Cedex 01, France}
\email{nicolas.raymond@univ-angers.fr}

\subjclass[2010]{}

\thanks{N. R. and F. H. are deeply grateful to the Mittag-Leffler Institute where part of the ideas of this article were discussed. N. R. also thanks Bernard Helffer, Pierig Keraval and Johannes Sjöstrand for many stimulating discussions.}

\date{}

\begin{abstract}
 The semiclassical magnetic Neumann Schr\"odinger operator on a smooth, bounded, and simply connected domain $\Omega$ of the Euclidean plane is considered. When $\Omega$ has a symmetry axis, the semiclassical splitting of the first two eigenvalues is analyzed. The first explicit tunneling formula in a pure magnetic field is established. The analysis is based on a pseudo-differential reduction to the boundary and the proof of the first known optimal purely  magnetic Agmon estimates.

\end{abstract}

\maketitle


\section{Introduction}

\subsection{A long-term investigation}

\subsubsection{The magnetic Laplacian with Neumann boundary condition}
Consider $\Omega$ a smooth, open, and simply-connected set of the plane. This article is devoted to the spectral analysis of the magnetic Laplacian $\mathscr{L}_h$ defined as the self-adjoint operator associated with the quadratic form
$$ \mathscr{Q}_h(\psi)=\int_\Omega |(-ih\nabla-\mathbf{A})\psi|^2\dd x\,.$$
 defined for $\psi \in H^1_{\mathbf{A}}(\Omega) \subset L^2(\Omega)$, the set for which $\mathscr{Q}_h(\psi)$ is finite.
In this article, the magnetic field is $B=\nabla\times\mathbf{A}=1$ and, by gauge invariance, we can choose $\mathbf{A}=(0,-x_1)$. The domain of $\mathscr{L}_h$ is
\begin{multline*}
\mathrm{Dom}(\mathscr{L}_h)=\big\{\psi\in H^1_{\mathbf{A}}(\Omega) : (-ih\nabla-\mathbf{A})^2\psi\in L^2(\Omega)\,,\\
\mathbf{n}\cdot(-ih\nabla-\mathbf{A})\psi=0 \mbox{ on }  \Gamma=\partial\Omega\big\}\,,
\end{multline*}
where $\mathbf{n}$ is the outward pointing normal to the boundary. In this paper, $L$ will denote the half-length of the boundary.

\subsubsection{From superconductivity to semiclassical analysis}
The original motivation to study the spectrum of $\mathscr{L}_h$ is the mathematical study of superconductivity. In particular, the asymptotic description of the third critical field (in the large magnetic field limit) is related to the groundstate energy of $\mathscr{L}_h$. For an overview of this vast subject, the reader is referred to the book \cite{FH10}. Independently of superconductivity, the subject has acquired a life of its own (see the book \cite{Ray}). Let us only point out some contributions directly related to the present framework. In \cite{HM01}, the ground state energy is analyzed and the following asymptotic formula is established
\begin{equation}\label{eq.hm}
\lambda_1(h)=\Theta_0 h-C_1\kappa_{\max}h^{\frac 32}+o(h^{\frac 32})\,,
\end{equation}
where $\kappa_{\max}$ is the maximum of the curvature of $\Gamma$, and $\Theta_0\in(0,1)$ and $C_1>0$ are related to the de Gennes operator (see \cite[Appendix A]{HM01}). This operator is defined as follows.
Consider, for all $\xi\in\R$, $\mathfrak{L}_\xi$ the Neumann realization on $\R_+$ of the operator $D^2_t+(\xi-t)^2$. The eigenvalues of $\mathfrak{L}_\xi$ are simple and denoted by $(\mu_n(\xi))_{n\geq 1}$. It is known (see \cite{DH93}) that $\mu_1$ has a unique and non-degenerate minimum at some $\xi_0>0$. We will denote by $u_\xi$ the positive $L^2$-normalized ground state. Then,
\begin{equation}\label{eq.theta0}
\Theta_0=\min_{\xi\in\R} \mu_1(\xi)\,,\quad C_1=\frac{u^2_{\xi_0}(0)}{6}\,.
\end{equation}
In relation with \eqref{eq.hm}, Helffer and Morame also proved that the first eigenfunctions are somehow localized near the boundary points of maximal curvature (see \cite[Theorem 10.6]{HM01} and the numerical simulation of the ground state when $\Omega$ is an ellipse, Figure \ref{fig.ev}). In contrast with \cite{HM01} where only the groundstate energy is considered, in \cite{FH06}, all the low lying eigenvalues are considered in the semiclassical limit when the curvature has a unique and non-degenerate minimum. Fournais and Helffer establish that, for all $n\geq 1$,
\begin{equation}\label{eq.HM}
\lambda_n(h)=\Theta_0 h-C_1\kappa_{\max}h^{\frac 32}+(2n-1)C_1\Theta_0^{\frac 14}\sqrt{\frac{3k_2}{2}}h^{\frac 74}+o(h^{\frac 74})\,,
\end{equation}
with $k_2=-\kappa''(s_0)$ where $\kappa$ is the curvature as a function of the curvilinear coordinate and $s_0$ the point of maximal curvature.
\begin{figure}[ht!]
	\includegraphics[width=10cm]{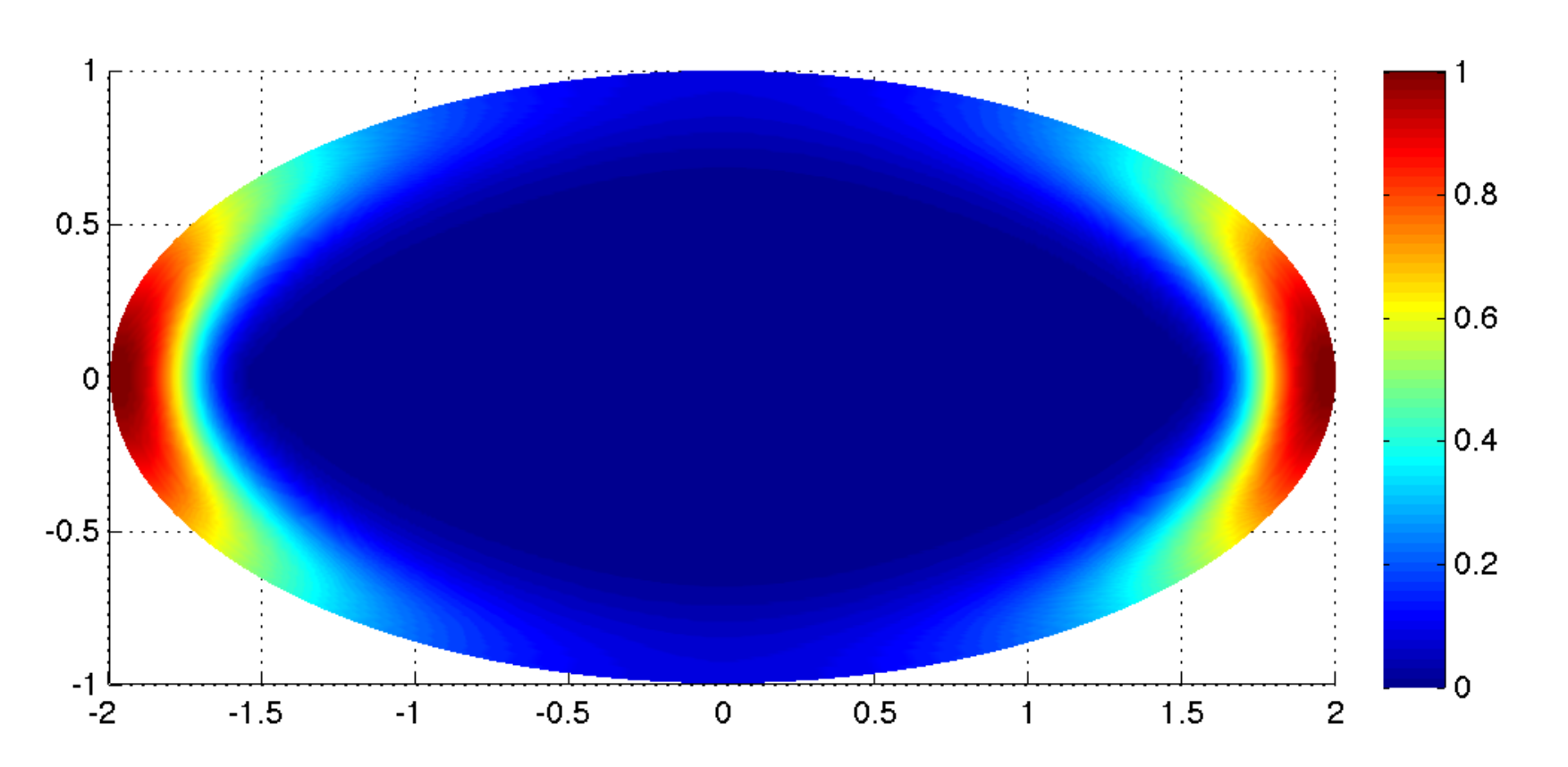}
	\caption{Modulus of the ground state when $\Omega$ is an ellipse. }
 \label{fig.ev}
\end{figure}

\subsubsection{Magnetic WKB constructions}
 In relation with \eqref{eq.HM},  we may wonder  how the corresponding  eigenfunctions behave and if we can accurately describe them in the semiclassical limit. It has been an open question for many years to know if the eigenfunctions could be written in a WKB form. A positive and very explicit answer has been given in \cite{BHR16} (see also Section \ref{sec.WKB} where we recall the result). It turned out that the magnetic operator is deeply connected to an effective electric operator acting on the boundary. Letting
\[\mathfrak{v}(s)=C_1(\kappa_{\max}-\kappa(s))\geq 0\,,\]	
the analysis there revealed the crucial role of the following effective eikonal equation
\begin{equation}\label{eq.eikonale}
\mathfrak{v}(s)-\frac{\mu_1''(\xi_0)}{2}\varphi'(s)^2=0\,.
\end{equation}

\subsubsection{An effective eikonal equation}
The remarkable feature of the aforementionned WKB analysis is that the eikonal equation \eqref{eq.eikonale} is the same, up to a local change of gauge, as the one obtained when considering the following \emph{purely electric Hamiltonian} acting on $L^2(\R/(2L\Z))$,
\begin{equation*}
\mathscr{L}_h^{\rm eff}=\frac{\mu_1''(\xi_0)}{2}\left(h^{\frac 12} D_s^2+V(s)\right)\,,\quad V(s)=\frac{2\mathfrak{v}(s)}{\mu''_1(\xi_0)}\,.
\end{equation*}
Let us denote by $(\lambda^{\mathrm{eff}}_n(h))_{n\geq 1}$ the sequence of its eigenvalues.

If $\mathfrak{v}$ has exactly two symmetric non-degenerate minima at $s_r\in(-L,0)$ and $s_\ell\in(0,L)$, it is well-known that the low lying spectrum is made of exponentially close pairs of eigenvalues. In order to describe the corresponding tunneling formula, we consider
 \begin{equation}\label{defS}
\mathsf{S} =\min \left(\mathsf{S}_{\mathsf{u}},\mathsf{S}_{\mathsf{d}}\right)\,,\quad  \mathsf{S}_{\mathsf{u}}=\int_{[s_{r},s_{\mathsf{\ell}}] } \sqrt{V(s)} \,ds\,,\quad  \mathsf{S}_{\mathsf{d}}=\int_{[s_{\mathsf{\ell}}, s_{\mathsf{r}}] } \sqrt{V(s)} \,ds\,,
\end{equation}
where $[p,q]$ denotes the arc joining $p$ and $q$ in the \enquote{circle} $\R/(2L\Z)$ counter-clockwise. The indices $\mathsf{u}$ and $\mathsf{d}$ refer to the up and down parts of the \enquote{circle} (corresponding to the up and down parts of $\partial\Omega$).

The tunneling formula is
 \begin{equation}\label{formspl1}
\lambda_2^{\rm eff}(h)-\lambda_1^{\rm eff}(h)
=2|w(h)| +\mathscr{O}(h^{\frac{3}{8}}e^{-\mathsf{S}/h^{1/4}})\,,
\end{equation}
where
\begin{equation}\label{eq:we}
w(h)= \mu_1''(\xi_0) h^{\frac{1}{8}} \pi^{-\frac 12} g^{\frac12}  \left(\mathsf{A}_{\mathsf{u}} \sqrt{V(0)}e^{- \mathsf{S}_{\mathsf{u}}/h^{1/4}}+\mathsf{A}_{\mathsf{d}} \sqrt{V(L)}e^{- \mathsf{S}_{\mathsf{d}}/h^{1/4}}\right)\,,
\end{equation}
with
\begin{equation}\label{eq.Aud}
\begin{split}
\mathsf{A}_{\mathsf{u}}&=\exp\left(-\int_{[s_{r}, 0]} \frac{ (V^\frac 12 )' (s)+g}{ \sqrt{V(s)}} ds\right)\,,\\
\mathsf{A}_{\mathsf{d}}&=\exp\left(-\int_{[s_{\mathsf{\ell}}, L]} \frac{ (V^\frac 12 )' (s) -g}{ \sqrt{V(s)}} ds\right)\,,\\
g&=\left(V''(s_{r})/2\right)^\frac 12=\left(V''(s_{\mathsf{\ell}})/2\right)^\frac 12\,.
\end{split}
\end{equation}
Such a one dimensional result goes back to \cite{Harrell}. This formula may also be found in \cite{BHR17} up to a convenient rescaling. The reader might also want to consider the Bourbaki exposé \cite{Robert} based on the celebrated Helffer-Sjöstrand theory developped in \cite{HS84, HS85a, HS85b, HS85c, HS86, HS87a, HS87b} (see also the series of works by Simon \cite{S1, S2, S3, S4}). In a periodic framework, flux effects are considered in \cite{O87} (see also \cite{BHR17}).

\subsubsection{Numerical simulations and conjecture}
More than a decade ago, the first numerical simulations describing magnetic tunneling effects in two dimensions appeared (see for instance \cite{BDMV07} in the case of corner domains). For instance, in the case of the ellipse (see Figure \ref{fig.L2-L1}), it was rather a surprise to be able to estimate an exponentially small effect and also to reveal the \enquote{oscillation} of $\lambda_2(h)-\lambda_1(h)$, numerically.
\begin{figure}[ht!]
	\includegraphics[width=9cm]{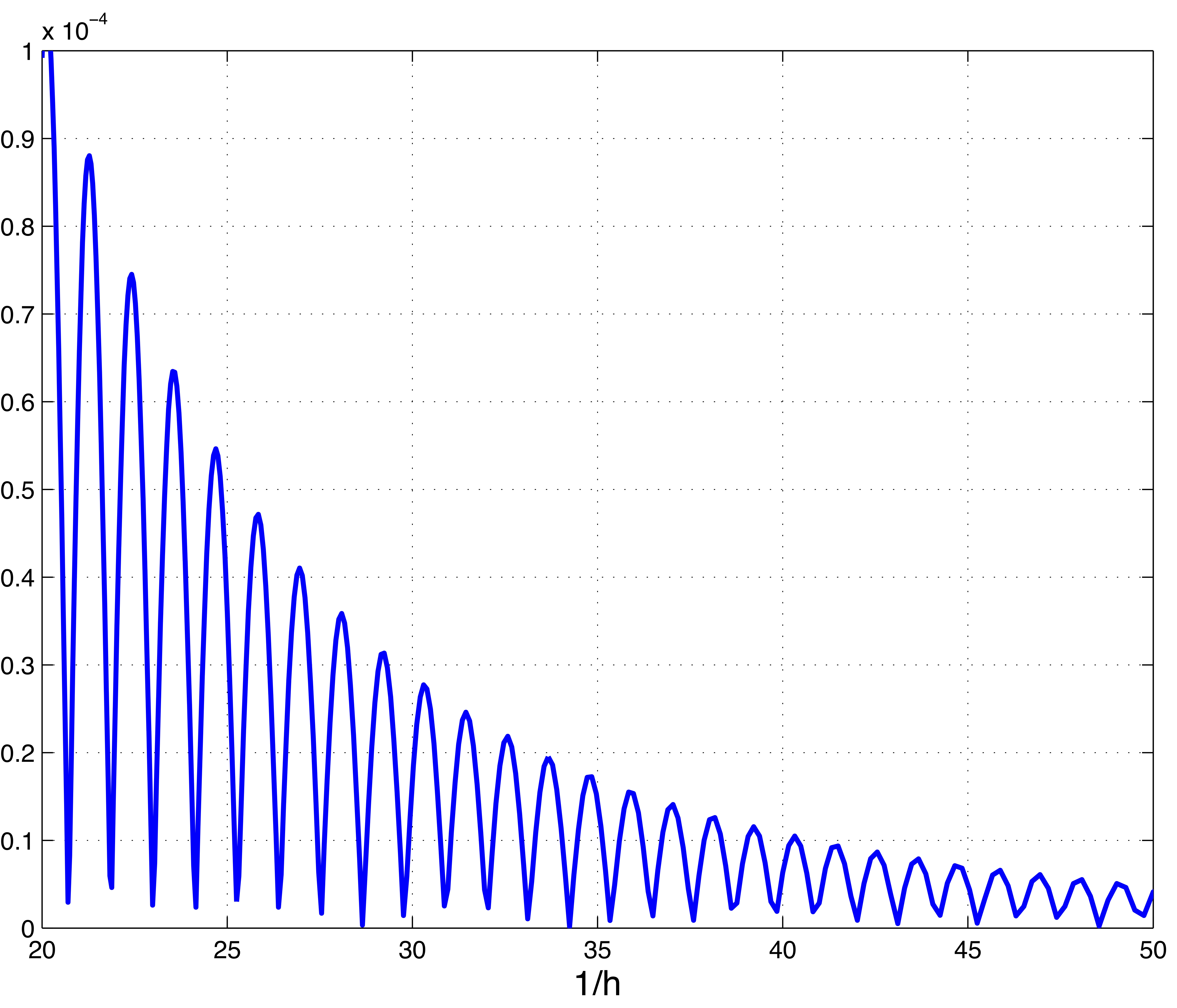}
	\caption{$\lambda_2(h)-\lambda_1(h)$ as a function of $1/h$ in the case of the ellipse}\label{fig.L2-L1}
\end{figure}

With these numerical computations arose the following open question:
\begin{center}
	Is there a theoretical formula to explain Figure \ref{fig.L2-L1}?	
\end{center}

For more numerical simulations concerning smooth domains with symmetries, the reader may consult \cite[Section 5.3.3]{BHR16} where \enquote{camels} (see Figure \ref{fig.cam}) and ellipses are considered. The case of varying (and vanishing) magnetic fields is also investigated.
\begin{figure}[ht!]
	\includegraphics[width=6.5cm]{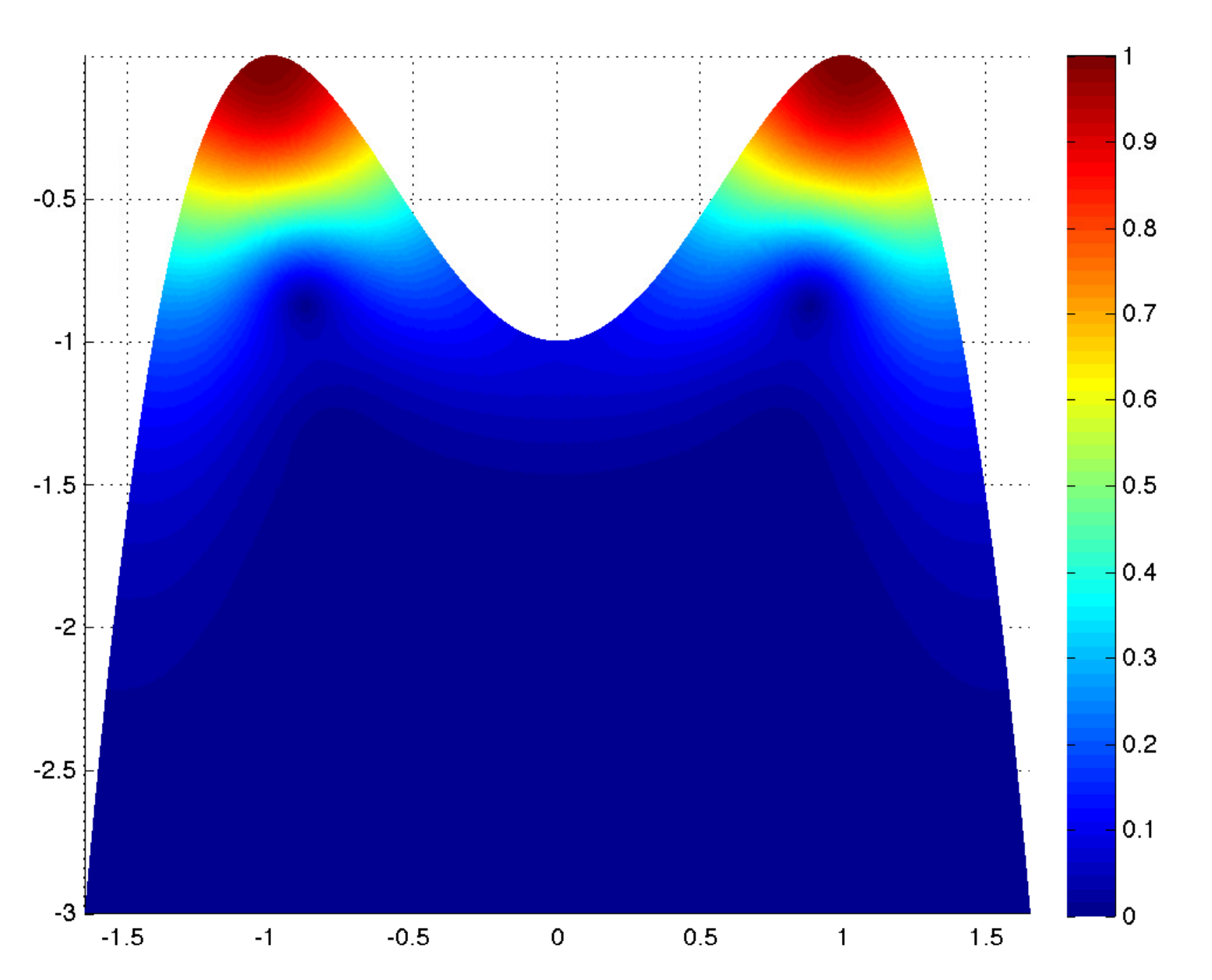}\hspace{1cm}\includegraphics[width=6.5cm]{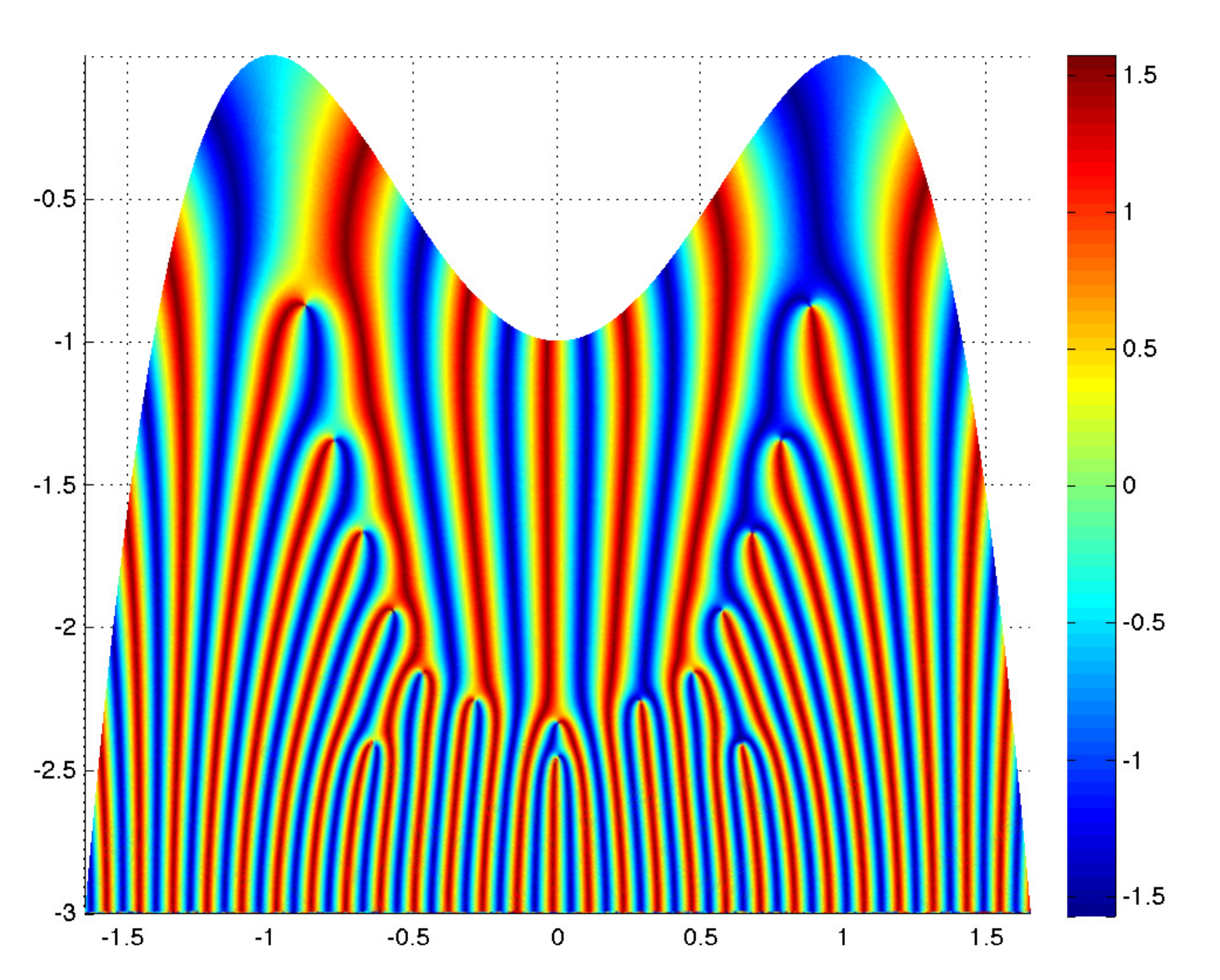}
	\caption{Modulus and phase of the groundstate in a camel-like domain}\label{fig.cam}
\end{figure}

Based on the WKB analysis in pure magnetic fields and the ideas \emph{à la} Born-Oppenheimer developped in \cite{BHR16}, we end up with the conjecture \cite[Conjecture 1.4]{BHR16b} of an \emph{explicit} formula to describe a \emph{purely magnetic tunneling} when $\Omega$ is an ellipse. This conjecture has been numerically checked (see Figure \ref{fig.conj}) and, to the authors' knowledge, is the first of its kind.

\begin{figure}[ht!]
	\includegraphics[width=9cm]{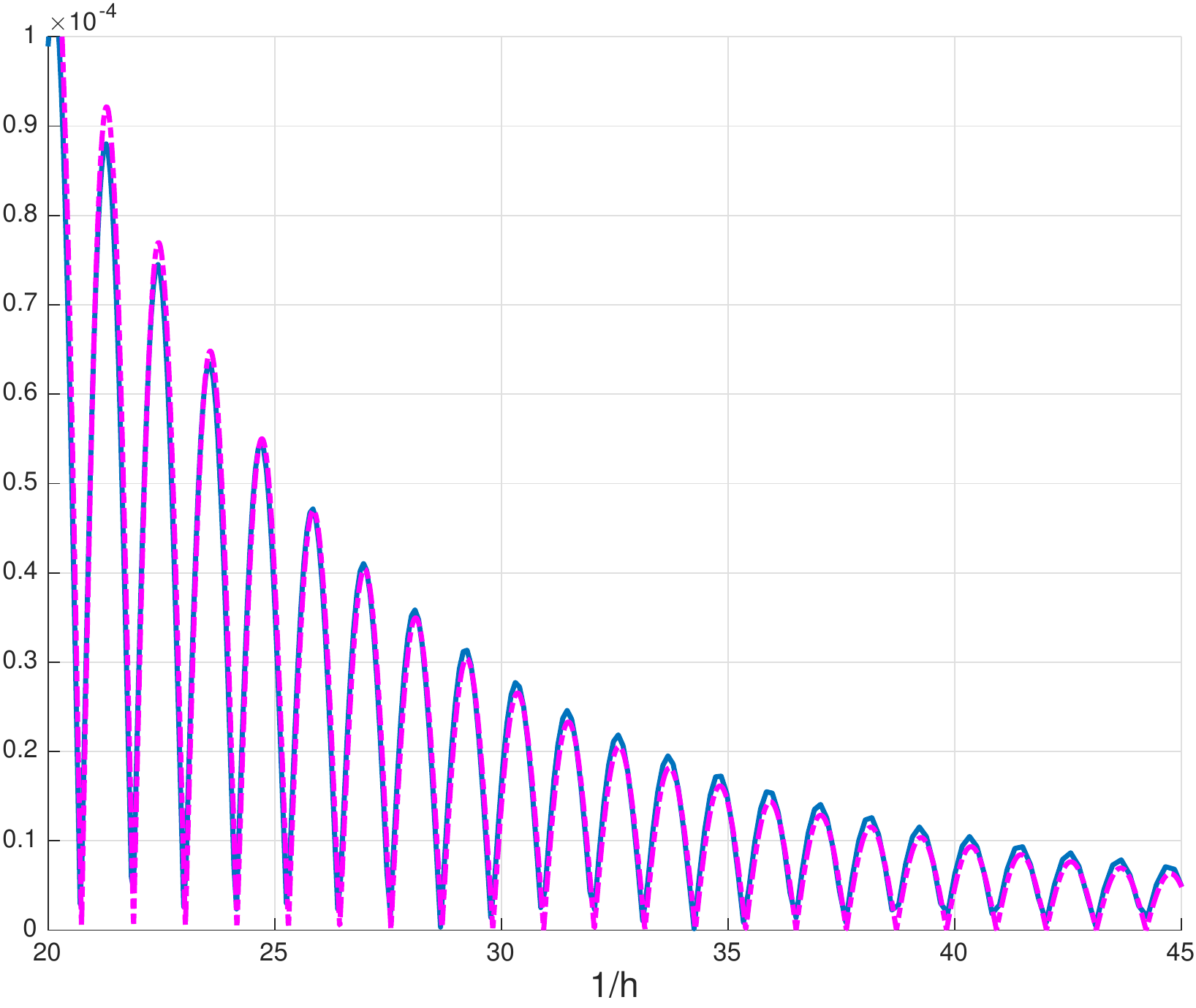}
	\caption{$\lambda_2(h)-\lambda_1(h)$ as a function of $1/h$; numerical simulation (blue) vs our conjecture (dashed)}\label{fig.conj}
\end{figure}
Let us recall this conjecture.

\begin{conjecture}\label{conj.0}
	Assume that $\Omega$ is an ellipse. Then, there exists $\alpha_0\in\R$ such that
	\begin{equation*}
	\begin{split}
	&\lambda_{2}(h)-\lambda_{1}(h)\\
	\underset{\hbar\to 0}{=}&
	h^{\frac{13}{8}}  \mathsf{A}\frac{2^{\frac{5}{2}} C_1^{\frac{3}{4}}}{\sqrt{\pi}} \left(k_2\mu_1''(\xi_0)\right)^{\frac{1}{4}}
	\left(\kappa_{\max}-\kappa_{\min}\right)^{\frac{1}{2}}
	 \times
	\left|\cos\left(L\left(\frac{\gamma_{0}}{h}-\frac{\xi_{0}}{h^{\frac12}}-
	\alpha_0\right)\right) \right|
	\mathrm{e}^{-{\mathsf{S}}/{h^{\frac{1}{4}}}}\\
	&+o(h^{\frac{13}{8}})\mathrm{e}^{-{\mathsf{S}}/{h^{\frac{1}{4}}}}\,,
	\end{split}
	\end{equation*}
	where
	\begin{equation*}
	\begin{split}
		\mathsf{S}&=
		\sqrt{\frac{2C_{1}}{\mu_1''(\xi_{0})}}\int_{-\frac L2}^{\frac L2}\sqrt{\kappa_{\max}-\kappa(s)}\,{\rm d} s\,,\\
		\mathsf{A} &=\exp\left(-\int_{[\frac L2, L]}\frac{\partial_{s}\sqrt{\kappa_{\max}-\kappa(s)}-\sqrt{\frac{k_2}{2}}}{\sqrt{\kappa_{\max}-\kappa(s)}}\,{\rm d} s\right)\,,\\
		\gamma_0&=\frac{|\Omega|}{|\Gamma|}=\frac{|\Omega|}{2L}\,,\quad k_2=-\kappa''\left(\frac L2\right)\,.
		\end{split}
	\end{equation*}
	Here, $s$ denotes the curvilinear coordinate. The points $s=-\frac L2$ and $s=\frac L2$ correspond to the right point of maximal curvature and to the left point of maximal curvature, respectively.
\end{conjecture}

The present article proves Conjecture \ref{conj.0} and, consequently, establishes the first explicit formula describing a purely magnetic tunneling effect.

\subsection{Statement of the  general  result}
Let us describe the geometric context of this article.

	\begin{center}
	\begin{figure}[ht!]
		\vspace*{-2cm}
		\includegraphics[width=11cm, angle=-90]{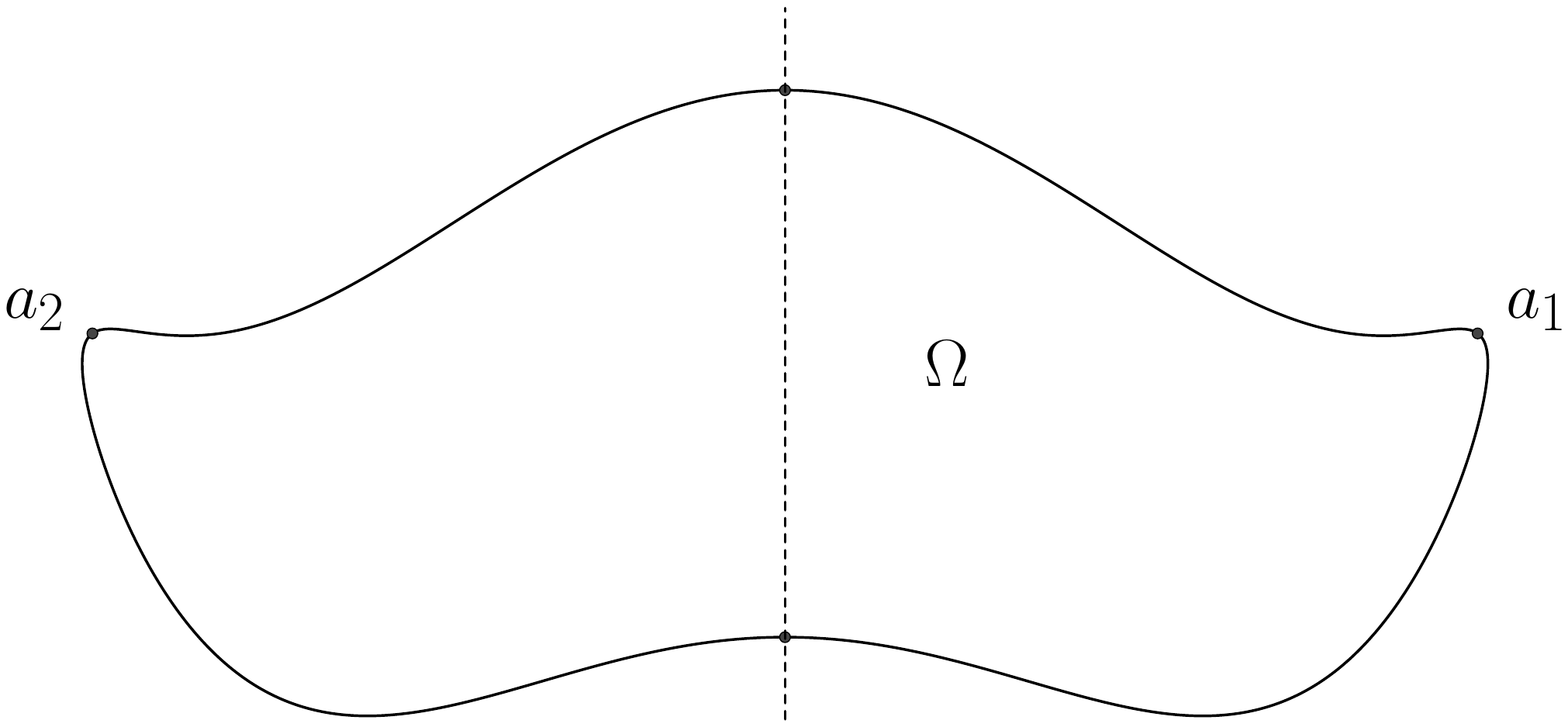}
		\vspace*{-4cm}
		\caption{A domain $\Omega$ with two symmetric curvature wells}\label{fig.tw}
	\end{figure}
\end{center}

\begin{assumption}\label{hyp.main}
$\Omega$ is a smooth, open, bounded, connected, and simply-connected set of the plane. Moreover, it is assumed that $\Omega$ is symmetric and that the curvature has two non-degenerate maxima:

\begin{enumerate}[\rm i)]
	\item $\Omega$ is symmetric with respect to the $y$-axis.
	\item The curvature $\kappa$ on the  boundary $\Gamma$ attains its maximum at exactly two points $a_1$ and $a_2$ which are not on the symmetry axis and belong to the same connected component of the boundary. We write
	\[a_1=(a_{1,1},a_{1,2})\in\Gamma\quad\mbox{ and }\quad a_2=(a_{2,1}, a_{2,2})\in\Gamma\,,\]
	such that $a_{1,1}>0$ and $a_{2,1}<0\,$.
	\item  The second derivative of the curvature (w.r.t. arc-length) at $a_1$ and $a_2$ is negative.
\end{enumerate}
\end{assumption}

We can now state the main theorem of this article, which gives, to the authors' knowledge, the first
optimal purely magnetic tunneling estimate.

\begin{theorem}\label{thm.tunnel}
Under Assumption \ref{hyp.main}, we have the tunneling formula
\[\lambda_2(h)-\lambda_1(h)=2|\tilde w(h)|+o(h^{\frac{13}{8}}e^{-\mathsf{S}/h^{\frac 14}})\,,\]
where
\begin{equation*}
\tilde w(h)= \mu_1''(\xi_0) h^{\frac{13}{8}} \pi^{-\frac 12} g^{\frac12}
\left(\mathsf{A}_{\mathsf{u}} \sqrt{V(0)}e^{- \mathsf{S}_{\mathsf{u}}/h^{1/4}} e^{iL f(h)}+\mathsf{A}_{\mathsf{d}} \sqrt{V(L)}e^{- \mathsf{S}_{\mathsf{d}}/h^{1/4}} e^{-iLf(h)}\right)\,,
\end{equation*}
where $V(s)=\frac{2C_1(\kappa_{\max}-\kappa(s))}{\mu''_1(\xi_0)}$ and
\begin{enumerate}[\rm i.]
\item $f(h)=\gamma_0/h-\xi_0/h^{1/2}-\alpha_0$,
\item $\alpha_0$ is a constant involving the de Gennes operator and the geometry (see \eqref{eq.alpha0}),
\item $\mathsf{A}_{\mathsf{u}}$, $\mathsf{A}_{\mathsf{d}}$ and $g$ are defined in \eqref{eq.Aud}.
\end{enumerate}
\end{theorem}

\begin{remark}
Let us make some remarks about Theorem \ref{thm.tunnel}. The proof actually allows to consider slightly  more general situations.
\begin{enumerate}[\rm i)]
\item Theorem \ref{thm.tunnel} implies Conjecture \ref{conj.0}. In the case of the ellipse, we have $s_{r}=-s_{\ell}=-\frac{L}{2}$ and $\kappa(0)=\kappa(L)=\kappa_{\min}$. Moreover, due the additionnal symmetry with respect to the horizontal axis, we have $\mathsf{A}_{\mathsf{u}}=\mathsf{A}_{\mathsf{d}}$ (see \eqref{eq.Aud}) and $\mathsf{S}_{\mathsf{u}}=\mathsf{S}_{\mathsf{d}}$ (see \eqref{defS}). This additionnal symmetry is thus responsible for the presence of the cosine in Conjecture \ref{conj.0}.
\item The assumption that $\Omega$ is bounded is not necessary to establish a tunneling result. Our strategy also applies to deal with camel-like domains (see Figure \ref{fig.cam}). In this simpler case, the \enquote{down} part in the tunneling formula has to be removed. Then, there is only one interaction term and no global flux effects. In particular, no oscillation of $\lambda_2(h)-\lambda_1(h)$ occurs.
\item The assumption that $\Omega$ is simply-connected is not necessary. The possible holes only contribute to change the value of $\gamma_0$.
\item The fact that we consider the first two eigenvalues, or only a domain with only one symmetry, is just for the simplicity of the presentation. The same strategy provides us with tunneling estimates in multiple well situations since our method reduces the analysis to one dimension electric tunneling (up to phase shifts).
\item In \cite{S4}, Simon described the \enquote{flea on the elephant effect}. This effect occurs when the electric potential is slightly perturbed and/or when the symmetry is broken. In this case, the first two eigenfunctions end up living in separate wells. In our case, such a phenomenon could be described as well (if we perturb the geometry of the boundary). In the special case of the ellipse, the oscillating effect is due to the existence of two minimal geodesics connecting the two curvature wells: If we slightly perturb the boundary (by keeping the symmetry) in such a way that $\mathsf{S}_{\mathsf{u}}\neq  \mathsf{S}_{\mathsf{d}}$, this kills one of the minimal geodesics and the beautiful oscillating effect disappears.
\end{enumerate}
\end{remark}

\begin{remark}
The investigation will reveal the \emph{microlocal} nature of the tunelling estimate given in Theorem \ref{thm.tunnel}. It contrasts with the electric tunneling \emph{à la} Helffer-Sjöstrand, and even with recent contributions about purely geometric tunneling \cite{HKR17} and \cite{KR17} where microlocal analysis is absent.
\end{remark}

\subsection{Organization and strategy}

In Section \ref{sec.2}, we explain how the spectral analysis of $\mathscr{L}_h$ can be reduced to the one of an operator $\mathscr{L}_{h,\delta}$ on a tubular neighborhood of the boundary, see Proposition \ref{prop.redtub}. Then, $\mathscr{L}_{h,\delta}$ is written in the classical tubular coordinates $(s,t)\in \R/(2L\Z)\times (0,\delta)$ and rescaled in the transverse variable $t=\hbar\tau$, with $\hbar=h^{\frac 12}$. The spectral analysis is then reduced to the one of $\mathscr{N}_\hbar$, see Proposition \ref{prop.redNh}.

In Section \ref{sec.3}, we consider a \enquote{one well problem} by removing the left maximum and gluing an infinite strip. Then, the resulting operator $\mathscr{N}_{\hbar, r}$ can be interpreted as a pseudo-differential operator with operator valued symbol the principal symbol of which being the de Gennes operator. Such operators and their spectrum have been extensively studied by Martinez via Grushin reductions. A concise presentation can be found in \cite{M07}. More details and extensions may also be found in the Ph. D. thesis of Keraval \cite{Keraval}. To some extent, our presentation will be similar to \cite{Martinez89} where tunneling estimates are provided in the case of partially semiclassical electric operators. In order to construct a parametrix of $\mathscr{N}_{\hbar, r}$\footnote{and actually of the conjugated operator  $\mathscr{N}^{\varphi}_{\hbar, r}= e^{\varphi/\hbar^{\frac 12}}\mathscr{N}_{\hbar, r}e^{-\varphi/\hbar^{\frac 12}}$, where $\varphi$ is an appropriate subsolution of the effective eikonal equation.}, one will need a convenient symbol class, see Notation \ref{not.S}. For that purpose, we will use a microlocal cutoff function and construct a parametrix for the \enquote{microlocalized} operator $\Op p_\hbar$ (near $\xi_0$), see Theorem \ref{thm.Grushin}.

In Section \ref{sec.4}, we use the parametrix to show that tangential elliptic estimates for $\mathscr{N}^{\varphi}_{\hbar, r}$ may be deduced from the one of an effective pseudo-differential operator acting on the boundary, see Theorem \ref{thm.coercivity}.

In Section \ref{sec.5}, we establish Theorem \ref{thm.coercivity-2D}. It is devoted to remove the frequency cutoff function introduced in Section \ref{sec.cutoff} up to using the transverse Agmon estimates, and the behavior at infinity of the de Gennes function $\mu_1$.

In Section \ref{sec.6}, we explain how to deduce optimal tangential Agmon estimates from Theorem \ref{thm.coercivity-2D} (see Corollary \ref{cor.Agmons}). We also establish slightly rougher tangential estimates for the \enquote{double well operator} $\mathscr{N}_\hbar$ from the one well estimates, see Proposition \ref{prop.Agmond}.

Section \ref{sec.7} is devoted to the proof of Theorem \ref{thm.tunnel}. We construct an approximate basis from the WKB Ansätze attached to each curvature well and compute the spectrum of the \emph{interaction} matrix thanks to the accurate WKB approximation of the ground state in each simple well.

\section{A reduction to a tubular neighborhood of the boundary}\label{sec.2}

\subsection{Normal Agmon estimates and spectral consequence}
The following proposition is well-known (see \cite[Theorem 4.1]{FH06}). It comes from the fact that the magnetic Laplacian on $\Omega$ with \emph{Dirichlet} boundary condition is bounded from below by $h$ since
\[\forall\psi\in\mathscr{C}^\infty_0(\Omega)\,,\quad \int_\Omega|(-ih\nabla-\mathbf{A})\psi|^2\dd x\geq h\int_\Omega|\psi|^2\dd x\,.\]
\begin{proposition}\label{prop.Agmont}
Let $M>0$. There exist $C, h_0, \alpha>0$ such that, for all $h\in(0,h_0)$, and all eigenpairs $(\lambda,\psi)$ of $\mathscr{L}_h$ with $\lambda\leq\Theta_0 h+Mh^{\frac 32}$,
\[\int_\Omega e^{2\alpha \mathrm{dist}(x,\Gamma)/h^{\frac 12}}|\psi|^2\dd x\leq C\|\psi\|^2\,,\]
and
\[\int_\Omega e^{2\alpha \mathrm{dist}(x,\Gamma)/h^{\frac 12}}|(-ih\nabla-\mathbf{A})\psi|^2\dd x\leq Ch\|\psi\|^2\,.\]	
\end{proposition}
This proposition tells us that the first eigenfunctions of $\mathscr{L}_h$ are exponentially localized in a neighbrohood of size $h^{\frac 12}$ of $\Gamma$. This invites us to define the new operator $\mathscr{L}_{h,\delta}$. Consider the (possibly $h$-dependent) $\delta$-neighborhood of the boundary
\[\Omega_\delta=\{x\in\Omega : \mathrm{dist}(x,\Gamma)<\delta\}\,.\]
(The dependence of $\delta$ w.r.t. $h$ will be precised later.)
Then, consider $\mathscr{L}_{h,\delta}$ the self-adjoint realization of $(-ih\nabla-\mathbf{A})^2$ with the following boundary conditions
\[\mathbf{n}\cdot(-ih\nabla-\mathbf{A})\psi=0\,,\mbox{ on }\Gamma\,,\]
and
\[\psi=0\,,\mbox{ on } \{x\in\Omega : \mathrm{dist}(x,\Gamma)=\delta\}\,,\]
where $\delta < \delta_0$ with $\delta_0$ small enough to ensure the smoothness of the boundary of $\Omega_\delta$.
The quadratic form $\mathscr{Q}_{h,\delta}$ associated with $\mathscr{L}_{h,\delta}$
is defined for all $\psi\in \mathcal{V}_\delta$,
\[\mathscr{Q}_{h,\delta}(\psi)=\int_{\Omega_\delta}|(-ih\nabla-\mathbf{A})\psi|^2\dd x\,,\]	
with
\[\mathcal{V}_\delta=\{\psi\in H^1(\Omega_\delta) : \psi(x)=0\,,\mbox{ on }\{x\in\Omega : \mathrm{dist}(x,\Gamma)=\delta\}\}\,.\]	
The operator $\mathscr{L}_{h,\delta}$ still has a compact resolvent and we can consider the non-decreasing sequence of its eigenvalues $(\lambda_{n}(h,\delta))_{n\geq 1}$ repeated according to their multiplicity.

\begin{proposition}\label{prop.redtub}
Let $n\geq 1$. There exist $C,h_0,\alpha>0$ such that, for all $h\in(0,h_0)$ and  $\delta \in (0,\delta_0) $,
\[\lambda_n(h)\leq\lambda_n(h,\delta)\leq \lambda_n(h)+Ce^{-\alpha\delta/h^{\frac 12}}\,.\]
\end{proposition}
\begin{proof}
The first inequality follows from the fact that $\Omega_\delta\subset\Omega$, the Dirichlet condition and the min-max principle. The second inequality follows from the Agmon estimates. Indeed, consider an orthonomal family of eigenfunctions $(\psi_{j})_{1\leq j\leq n}$ associated with $(\lambda_j(h))_{1\leq j\leq n}$ and let
\[\mathscr{E}_n(h,\delta)=\underset{1\leq j\leq n}{\mathrm{span}}\, \chi_\delta\psi_j\,.\]
Here $\chi_\delta$ is defined by $\chi_\delta(x)=\chi\left(\frac{\mathrm{dist}(x,\Gamma)}{\delta}\right)$ where $\chi$ is a smooth function such that $\chi(x)=1$ for $x\in [0,1/2)$ and $\chi(x)=0$ for $x\geq 1$. Thus, $\mathscr{E}_n(h,\delta)\subset\mathcal{V}_\delta$. Consider $\tilde\psi\in \mathscr{E}_n(h,\delta)$ and write
\[\tilde\psi=\chi_\delta\psi=\chi_\delta\sum_{j=1}^n\beta_j\psi_{j}\,.\]
We have
\begin{equation*}
\begin{split}
\mathscr{Q}_{h,\delta}(\chi_\delta\psi)&=\int_{\Omega}|\chi_\delta(-ih\nabla-\mathbf{A})\psi-ih\psi\nabla\chi_\delta |^2\dd x\\
&\leq \|(-ih\nabla-\mathbf{A})\psi\|^2+2h\|(-ih\nabla-\mathbf{A})\psi\|_{L^2(\Omega\setminus\Omega_{\delta/2})}\|\psi\nabla\chi_\delta \|+h^2\|\psi\nabla\chi_\delta\|^2\,.
\end{split}
\end{equation*}
Then, since the $(\psi_j)_{1\leq j\leq n}$ are orthogonal eigenfunctions, we get
\[\|(-ih\nabla-\mathbf{A})\psi\|^2\leq \lambda_n(h)\|\psi\|^2\,.\]
From Proposition \ref{prop.Agmont}, we have
\[\|\psi \nabla\chi_\delta \|\leq C\delta^{-1}e^{-\alpha\delta/2h^{\frac 12}}\|\psi\|\,,\quad \|(-ih\nabla-\mathbf{A})\psi\|_{L^2(\Omega\setminus\Omega_{\delta/2})}\leq Ch^{\frac 12}e^{-\alpha\delta/2h^{\frac 12}}\|\psi\|\,.\]
It follows that
\[\mathscr{Q}_{h,\delta}(\chi_\delta\psi)\leq \left(\lambda_n(h)+C(h^{\frac 32}\delta^{-1}+h^2\delta^{-2})e^{-\alpha\delta/h^{\frac 12}}\right)\| \psi\|^2\,,\]
and then
\[\mathscr{Q}_{h,\delta}(\chi_\delta\psi)\leq \left(\lambda_n(h)+C(h+h^{\frac 32}\delta^{-1}+h^2\delta^{-2})e^{-\alpha\delta/h^{\frac 12}}\right)\|\chi_\delta\psi\|^2\,.\]
\end{proof}

\subsection{Tubular coordinates and truncated operator}
We will use the canonical tubular coordinates $(s, t)$ where $s$ is the arc-length and $t$ the distance to the boundary. We recall some elementary properties of these coordinates. Let
\begin{equation}\label{eq:s}
(-L,L]\ni s\mapsto M(s)\in\Gamma
\end{equation}
be a parametrization of $\Gamma$.  The unit tangent vector of $\Gamma$ at the point $M(s)$ of the boundary is given by
\[
T(s):= M^{\prime}(s).
\]
We define the  curvature $\kappa(s)$ by the following identity
\[
T^{\prime}(s)=-\kappa(s)\, \mathbf{n}(s),
\]
where $\mathbf{n}(s)$ is the unit vector, normal to the
boundary, pointing outward at the point $M(s)$. We choose the
orientation of the parametrization $M$ to be counter-clockwise, so
\[
\det(T(s),\mathbf{n}(s))=1, \qquad \forall s\in  (-L,L].
\]
We introduce the change of coordinates
\begin{equation}
\Phi:\mathbb{R}/((2L)\mathbb{Z})\times(0,\delta)\ni  (s,t)\mapsto x= M(s)-t\, \mathbf{n}(s)\in \Omega_{\delta}.
\end{equation}
The determinant of the Jacobian of $\Phi$ is
given by
\begin{equation}\label{eq.m}
m(s,t)=1-t\kappa(s).
\end{equation}
Thanks to this change of coordinates, $\mathscr{L}_{h,\delta}$ is unitarily equivalent to $\mathscr{M}_{h,\delta}$ the self-adjoint realization on $L^2(\Gamma\times (0,\delta), m\dd s\dd t)$, of the differential operator
\begin{equation*}
-h^2m^{-1}\partial_t m\partial_t+m^{-1}\left(-ih\partial_s+\gamma_0-t+\frac{\kappa}{2}t^2\right)m^{-1}\left(-ih\partial_s+\gamma_0-t+\frac{\kappa}{2}t^2\right)\,,
\end{equation*}
where
\[m(s,t)=1-t\kappa(s)\,,\quad \gamma_0=\frac{|\Omega|}{|\Gamma|}\,,\]
with the boundary conditions
\[\partial_t\psi(s,0)=0\,,\quad \psi(s,\delta)=0\,.\]
This fact can be found in \cite[Appendix F]{FH10}. The first eigenfunctions of $\mathscr{M}_{h,\delta}$ also satisfy Agmon estimates (with respect to $t$).

\begin{proposition}
Let $M>0$. There exist $C, h_0, \alpha>0$ such that, for all $h\in(0,h_0)$, and all eigenpair $(\lambda,\psi)$ of $\mathscr{M}_{h,\delta}$ with $\lambda\leq\Theta_0 h+Mh^{\frac 32}$,
\[\int_\Omega e^{2\alpha t/h^{\frac 12}}|\psi|^2\dd s\dd t\leq C\|\psi\|^2\,,\]
and
\[\int_\Omega e^{2\alpha t/h^{\frac 12}}\left(\left|(-ih\partial_s+\gamma_0-t-\frac{\kappa}{2}t^2)\psi\right|^2+|h\partial_t\psi|^2\right)\dd s\dd t\leq Ch\|\psi\|^2\,.\]	
\end{proposition}
These estimates invite us to consider an operator on the space domain $\Gamma\times (0,+\infty)$ instead of $\Gamma\times(0,\delta)$. For this we insert cutoff functions in the preceding operator. Let $c$ be a smooth real function equal to $1$ on $[0, 1]$ and $0$ for $t \geq 2$. Then, we let
\[\underline{m}(s,t)=1-tc(\delta^{-1}t)\kappa(s)\,.\]
Instead of $\mathscr{M}_{h,\delta}$, we consider
$\underline{\mathscr{M}}_{h,\delta}$ the self-adjoint realization on the Hilbert space $L^2(\Gamma\times (0,+\infty), \underline{m}\dd s\dd t)$, of the differential operator with associated eigenvalues $\underline{\lambda}_n(h,\delta)$.
\begin{multline*}
-h^2\underline{m}^{-1}\partial_t \underline{m}\partial_t\\
+\underline{m}^{-1}\left(-ih\partial_s+\gamma_0-t+c(\delta^{-1}t)\frac{\kappa}{2}t^2\right)\!\underline{m}^{-1}\left(-ih\partial_s+\gamma_0-t+c(\delta^{-1}t)\frac{\kappa}{2}t^2\right)\,,
\end{multline*}
with Neumann boundary condition on $t=0$. Note here that the additional truncation in front of $\kappa$ is introduced in order to make this term bounded (and later a lower order term) when $t$ is large.

Using the same truncation trick as in the proof of Proposition \ref{prop.redtub}, similar Agmon type estimates for $\underline{\mathscr{M}}_{h,\delta}$, and the min-max principle, we get the following.

\begin{proposition} \label{prop.redtubinter}
	Let $n\geq 1$. There exist $C,h_0,\alpha>0$ such that, for all $h\in(0,h_0)$ and  $\delta \in (0,\delta_0) $,
	\[\underline{\lambda}_n(h, \delta) \leq \lambda_n(h,\delta)\leq \underline{\lambda}_n(h, \delta)+ Ce^{-\alpha\delta/h^{\frac 12}}\,.\]
\end{proposition}	
 \begin{remark}

 Actually, at this stage, we have not proved that the low-lying spectrum of $\underline{\mathscr{M}}_{h,\delta}$ is discrete. This will be a consequence of the forthcoming analysis.

 \end{remark}
  From now on we fix
 \[\delta=h^{\frac 14-\eta}\gg h^{\frac 14}\,,\]
 for some fixed $0 < \eta < 1/4$.  Note that this assumption is sufficient to ensure that remainder terms appearing in the latter proposition are indeed controlled by the main term which is of order $e^{-S/h^{\frac 14}}$ for some constant $S$ (see the main statement in Theorem \ref{thm.tunnel}).

\subsection{The rescaled operator}
The exponential localization at the scale $h^{\frac 12}$ near $t=0$ suggests to consider the partial rescaling
\[(s,t)=(\sigma,\hbar\tau)\,,\quad\mbox{ with }\hbar=h^{\frac 12}\,.\]
We also let
\[a_\hbar(\sigma,\tau)=1-\hbar \tau\kappa(\sigma) c_\mu(\tau)\,,
\quad c_\mu(\tau)=c(\mu\tau) \quad \textrm{ for } \mu =\hbar^{\frac 12+2\eta}\,, \]
 where we recall that $\eta$ is positive and small, and $c$ is the cutoff function introduced in the preceding section.
\begin{remark}\label{rem.mu}
The notation $\mu$ will be convenient later when expanding the operator in powers of $\hbar$, with coefficients depending on $\mu$.
\end{remark}
Note that
\[
a_\hbar = 1+\mathscr{O}(\hbar^{\frac 12-2\eta})\,.
\]
Upon dividing $\underline{\mathscr{M}}_{h,\delta}$ by $h$, we get the new operator $\mathscr{N}_\hbar$ acting on the space $L^2(\Gamma\times\R_+, a_\hbar\dd s\dd t)= L^2(\Gamma\times \R_+, \dd s\dd t)$, as the differential operator
\begin{multline*}
\mathscr{N}_\hbar = -a_\hbar^{-1}\partial_\tau a_\hbar\partial_\tau \\
+a_\hbar^{-1}\left(-i\hbar\partial_\sigma+\hbar^{-1}\gamma_0-\tau+\hbar c_\mu\frac{\kappa}{2}\tau^2\right)a_\hbar^{-1}\left(-i\hbar\partial_\sigma+\hbar^{-1}\gamma_0-\tau+\hbar c_\mu\frac{\kappa}{2}\tau^2\right)
\end{multline*}
with Neumann  condition on $\tau=0$. 
Note that 
\[\begin{split}
&\Dom(\mathscr{N}_\hbar)=\\
\Big\{&u\in L^2(\Gamma\times\R_+) : -\partial^2_\tau u\in L^2(\Gamma\times \R_+)\,, \big(-i\hbar\partial_\sigma+\hbar^{-1}\gamma_0-\tau\big)^2u\in L^2(\Gamma\times \R_+)\,,\\ 
&\partial_\tau u(\cdot,0)=0 \Big\}\,.
\end{split}
\]
We denote by $(\nu_n(\hbar))_{n\geq 1}$ its eigenvalues. Using then Propositions \ref{prop.redtub} and \ref{prop.redtubinter}, we get

\begin{proposition}\label{prop.redNh}
Let $n\geq 1$. There exist $K>\mathsf{S}$, $C,h_0>0$ such that, for all $h\in(0,h_0)$,
\[\lambda_n(h)-Ce^{-K/h^{\frac 14}}\,\leq\hbar^2\nu_n(\hbar)\leq \lambda_n(h)+Ce^{-K/h^{\frac 14}}\,.\]	
\end{proposition}
This means that, in order to estimate the expected splitting between eigenvalues $\lambda_2(h)-\lambda_1(h)$ of the original operator, we can consider the corresponding splitting for the reduced and rescaled operator $\mathscr{N}_\hbar$ . The rest of the article is devoted to this problem.

\subsection{One well operators}

\subsubsection{Definitions}\label{sec.defonewell}
Let us consider the \enquote{one well operator} (attached to the right well). It is geometrically defined by surgery by removing a small neighborhood of the left curvature maximum, and gluing an infinite strip, see Figure \ref{fig.ow}. For this we choose first the curvilinear origin at the intersection of the upper part of $\Gamma$ and the vertical axis, and we  identify $\Gamma$ with $[s_\ell-2L, s_\ell]$.  Note that in these coordinates, we have $s_r< 0<s_\ell$. We consider then the following \it right well \rm differential operator
\begin{multline}\label{eq.Nhr}
 \mathscr{N}_{\hbar, r, \gamma_0}:= -a_\hbar^{-1}\partial_\tau a_\hbar\partial_\tau\\
+a_\hbar^{-1}\left(-i\hbar\partial_\sigma+\hbar^{-1}\gamma_0-\tau+\hbar c_\mu\frac{\kappa_r}{2}\tau^2\right)a_\hbar^{-1}\left(-i\hbar\partial_\sigma+\hbar^{-1}\gamma_0-\tau+\hbar c_\mu\frac{\kappa_r}{2}\tau^2\right)\,,
\end{multline}
acting on $L^2(\R\times\R_+, a_\hbar \dd \sigma \dd\tau)$ where $\kappa_r$ is an appropriate extension of $\kappa$ defined as follows:
\[\kappa_r=\kappa\,,\quad \mbox{ on }\quad I_{r,\eta}:=(s_\ell-2L+\eta, s_\ell-\eta)\,,\]
and $\kappa_r=0$ on $(-\infty, s_\ell-2L)\cup(s_\ell+\infty)$.
This extension may be chosen so that $\kappa_r$ has a unique and non-degenerate maximum at $s_r<0$.

\begin{center}
	\begin{figure}[ht!]
		\vspace*{-2cm}
		\includegraphics[width=12cm, angle=-90]{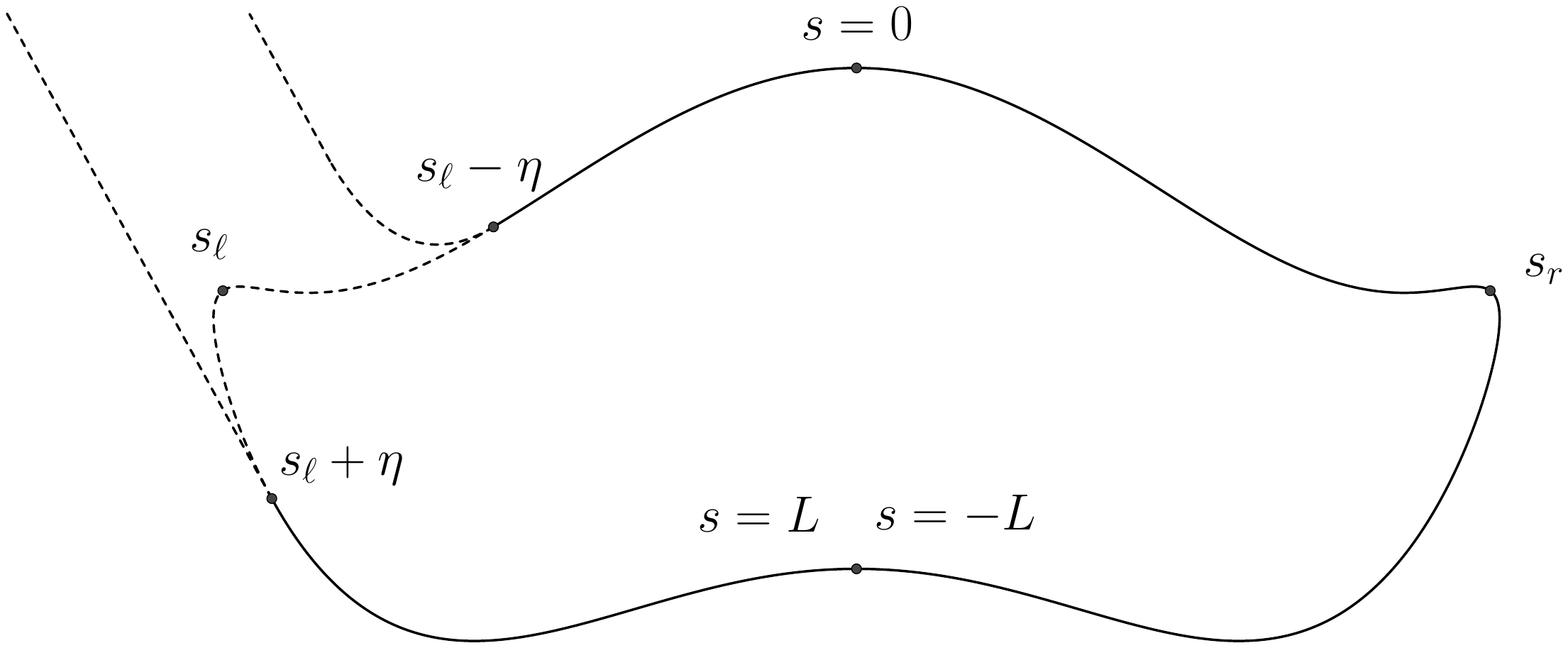}
		\vspace*{-3cm}
		\caption{One well domain attached to the right well}\label{fig.ow}
	\end{figure}
\end{center}
Since the space domain is now simply connected, $\mathscr{N}_{\hbar, r, \gamma_0}$ is unitarily equivalent to the flux-free operator $\mathscr{N}_{\hbar, r}:=\mathscr{N}_{\hbar, r, 0}$ since $e^{i\sigma\gamma_0/\hbar^2}\mathscr{N}_{\hbar, r, \gamma_0}e^{-i\sigma\gamma_0/\hbar^2}=\mathscr{N}_{\hbar, r, 0}$.
Note that the domain is the same as the one of the operator with constant magnetic field on $\R^2_+=\R\times\R_+$:
\[\begin{split}
&\Dom(\mathscr{N}_{\hbar,r})=\\
\Big\{&u\in L^2(\R^2_+) : -\partial^2_\tau u\in L^2(\R^2_+)\,, \big(-i\hbar\partial_\sigma-\tau\big)^2u\in L^2(\R^2_+)\,,\partial_\tau u(\cdot,0)=0 \Big\}\,.
\end{split}
\]
Let us now consider $u_{\hbar,r}$ a groundstate of the flux-free operator ${\mathscr{N}}_{\hbar, r, 0}$ with well at $s_r<0$. The bottom of the spectrum is indeed discrete. Let us briefly explain this. Firstly, the essential spectrum is $[\Theta_0,+\infty)$, since, at infinity with respect to $\sigma$, the operator coincides with the Neumann magnetic Laplacian on the half-plane, whose spectrum is $[\Theta_0,+\infty)$. Secondly, one knows (see Theorem \ref{thm.BKW}) that the spectrum below $\Theta_0$ is not empty.

The function
\begin{equation}\label{eq.phir0}
\check\phi_{\hbar, r}(\sigma,\tau) =
e^{-i\gamma_0 \sigma/\hbar^2}u_{\hbar,r}(\sigma,\tau)
\end{equation}
is then a groundstate for $\mathscr{N}_{\hbar, r, \gamma_0}$.

In order to define an operator adapted to the left well, we use the symmetry of $\Gamma$. More precisely, we
consider the symmetry operator
\[U f(\sigma,\tau)=\overline{f(-\sigma,\tau)}\,, \]
and define
$$
{\mathscr{N}}_{\hbar, \ell, \gamma_0}=U^{-1}{\mathscr{N}}_{\hbar, r, \gamma_0}U\,.
$$
Note that this operator also corresponds to the following construction. Identifying $\Gamma$ with $[s_r, s_r+ 2L]$, we can define on $\R$ the extended curvature $\kappa_\ell(\cdot) := \kappa_r(-\cdot)$ and note that it is equal to $\kappa$ on $(s_r+\eta, s_r+2L-\eta)$ and $0$  on $(-\infty, s_r)\cup(s_r+2L,+\infty)$.
In this way, $\kappa_\ell$ has a unique and non-degenerate maximum at $s_\ell>0$. The operator
 $\mathscr{N}_{\hbar, \ell, \gamma_0}$ acting on
$L^2(\R\times\R_+, a_\hbar \dd \sigma \dd\tau)$ with well at $s_\ell$ has then the same expression as the one of ${\mathscr{N}}_{\hbar, r, \gamma_0}$.
A natural groundstate for ${\mathscr{N}}_{\hbar, \ell, \gamma_0}$ is then
\[
\check\phi_{\hbar,\ell}:=U\check\phi_{\hbar,r}
\]
 and, letting $u_{\hbar,\ell}=Uu_{\hbar,r}$, we have
\begin{equation}\label{eq.phil0}
\check\phi_{\hbar, \ell}(\sigma,\tau) =
e^{-i\gamma_0 \sigma/\hbar^2}u_{\hbar,\ell}(\sigma,\tau)\,.
\end{equation}
In the following, we will focus on the right well and find a WKB approximation of $u_{\hbar,r}$.

\subsubsection{WKB construction}\label{sec.WKB}
The following fundamental theorem has been established in \cite[Theorem 5.6 \& Section 5.3.2]{BHR16}. Let us recall that
\[
V(s) =\frac{2C_1}{\mu''_1(\xi_0)}\sep{ \kappa_{\max}-\kappa_r(s)}\,.
\]
In what follows, we consider formal series in the sense of \cite[Notation 1.13]{BHR16}, where the neighbourhood on which the approximation (at any given order) occurs can be taken arbitrarily large but bounded.
 \begin{theorem}\label{thm.BKW}
Let us consider the following Agmon distance to the right well $s_r$:
\begin{equation} \label{eq.defPhi} \Phi(\sigma)=\int_{[s_r,\sigma]}\sqrt{V(\tilde\sigma)}\dd\tilde\sigma\,.
\end{equation}
There exist formal series $(b_n(\hbar)_{n\geq 0}$ and $(\delta_n(\hbar))_{n\geq 0}$ such that
\[b_n(\hbar)\sim\sum_{j\geq 0}b_{n,j}\hbar^{\frac j2}\,,\quad \delta_n(\hbar)\sim\sum_{j\geq 0}\delta_{n,j}\hbar^{\frac j2}\,,\]
and
\[\left(\mathscr{N}_{\hbar, r}-\delta_n(\hbar)\right)\Psi_{\hbar,r,n}=\mathscr{O}(\hbar^\infty)e^{-\Phi(\sigma)/\hbar^{\frac 12}}\,,\]
with
\begin{equation}\label{eq.psir}
\Psi_{\hbar,r,n}\underset{\hbar\to 0}{\sim}\hbar^{-\frac18} b_n(\hbar)e^{-\Phi(\sigma)/\hbar^{\frac 12}} e^{i\sigma\xi_0/\hbar}\,.
\end{equation}
Moreover,
\[\delta_{n,0}=\Theta_0\,,\quad \delta_{n,1}=0\,,\quad\delta_{n,2}=-C_1\kappa_{\max}\,,\quad\delta_{n,3}=(2n-1)C_1\Theta_0^{\frac 14}\sqrt{\frac{3k_2}{2}}\,,\]
and
\begin{equation}\label{eq.an0}
b_{n,0}(\sigma,\tau)=f_{n,0}(\sigma)u_{\xi_0}(\tau)\,,\quad b_{n,1}(\sigma,\tau)=i\Phi'(\sigma)f_{n,0}(\sigma)(\partial_\xi u_\xi)_{\xi_0}(\tau)+f_{n,1}(\sigma) u_{\xi_0}(\tau)\,,
\end{equation}
where $f_{n,0}$ solves the effective transport equation
\begin{equation}\label{eq.effectiveT}
\frac{\mu''_1(\xi_0)}{2}(\Phi'\partial_\sigma+\partial_{\sigma}\Phi')f_{n,0}+F(\sigma)f_{n,0}=(2n-1)C_1\Theta_0^{\frac 14}\sqrt{\frac{3k_2}{2}}f_{n,0}\,,
\end{equation}
$f_{n,1}$ is a solution of a similar transport equation, and where $F$ is a smooth function such that $F(s_r)=0$ and $\Re F=0$.

\end{theorem}

\begin{remark}\label{rem.tildef}
Let us consider \eqref{eq.effectiveT}. We may write $f_{n,0}$ in the form $f_{n,0}(\sigma)=e^{i\alpha_{n,0}(\sigma)}\tilde f_{n,0}(\sigma)$ for some real-valued function $\alpha_{n,0}$, and where $\tilde f_{n,0}$ solves the \emph{real} classical transport equation
\begin{equation}\label{eq.effectiveT2}
\frac{\mu''_1(\xi_0)}{2}(\Phi'\partial_\sigma+\partial_{\sigma}\Phi')\tilde f_{n,0}=(2n-1)C_1\Theta_0^{\frac 14}\sqrt{\frac{3k_2}{2}}\tilde f_{n,0}\,.
\end{equation}
Note that $\tilde f_{1,0}(0)$ be can chosen to be positive and we will assume that it is the case. Following e.g. \cite[Section 2.2]{BHR17} we also choose the normalization $\norm{\Psi_{\hbar,r,n}} = 1$ (see also Section  \ref{sec.apprx1well}). This gives 
\begin{equation}  \label{eq.normtildef10}
\tilde f_{1,0}^2(0) = \sep{\frac{g}{\pi}}^{1/2} \mathsf{A}_{\mathsf{u}}\,,
\end{equation}
where $g$ and $\mathsf{A}_{\mathsf{u}}$ are defined in \eqref{eq.Aud}.

Let us now consider the phase shifts.  The $\alpha_{n,0}$ are chosen so that
\[\mu_1''(\xi_0)i\Phi'\alpha'_{n,0}+F=0\,,\quad \mbox{or, equivalently,}\quad  \Phi'\alpha'_{n,0}=\frac{iF}{\mu''_1(\xi_0)}\,.\]
Since $F(s_r)=0$ and $\Phi'$ vanishes linearly at $s_r$, we can write
\begin{equation} \label{eq.alphanprime}
\alpha'_{n,0}(\sigma)=\frac{iF(\sigma)}{\mu''_1(\xi_0)\Phi'(\sigma)}\,,
\end{equation}
where the right-hand-side can be seen as a real-valued ($\Re F=0$) smooth function defined at $s_r$
(by using the natural continuous extension). This determines the phase shift $\alpha_{n,0}$ up to an additive constant.

At this stage, we fixed the normalization of the WKB Ansatz. Later on, in Section \ref{sec.apprx1well}, we will take profit of this appropriate normalization, which determines the functions $\tilde f_{1,0}$ and $\alpha_{1,0}$. This normalization of $\tilde f_{1,0}$ is the one that we used in \cite[Section 2.2]{BHR17} when considering the tunneling effect for purely electric Schrödinger operators on the circle. This will be suitable to recognize, in our final computation, the interaction term for an electric Hamiltonian. The equation \eqref{eq.effectiveT2} is indeed the same as the one we obtain when performing a WKB construction for the semiclasssical electric Hamiltonian
\begin{equation}\label{eq.effectiveH}
\frac{\mu_1''(\xi_0)}{2}\hbar D_\sigma^2+\mathfrak{v}(\sigma)\,,\quad \mathfrak{v}=C_1(\kappa_{\max}-\kappa_r)\,.
\end{equation}
We define
\begin{equation}\label{eq.alpha0}
\alpha_0=\frac{\alpha_{1,0}(0)-\alpha_{1,0}(-L)}{L}\,,
\end{equation}
which is the phase shift appearing in Theorem \ref{thm.tunnel}
\end{remark}

\section{A Grushin problem}\label{sec.3}
In this section, we focus on the one well operator. Let us consider a smooth non-negative function, with bounded derivative, $\sigma\mapsto\varphi(\sigma)$ and consider the conjugate operator
\[\mathscr{N}^\varphi_{\hbar,r}=e^{\varphi/\hbar^{\frac 12}}\mathscr{N}_{\hbar, r} e^{-\varphi/\hbar^{\frac 12}}\,,\]
still acting on $\Dom(\mathscr{N}_{\hbar, r})$. 

Explicitly,
\begin{multline*}
\mathscr{N}^\varphi_{\hbar,r}=-a_\hbar^{-1}\partial_\tau a_\hbar\partial_\tau\\
+a_\hbar^{-1}\left(-i\hbar\partial_\sigma-\tau+i\hbar^{\frac 12}\varphi'+\hbar c_\mu\frac{\kappa_r}{2}\tau^2\right)a_\hbar^{-1}\left(-i\hbar\partial_\sigma-\tau+i\hbar^{\frac 12}\varphi'+\hbar c_\mu\frac{\kappa_r}{2}\tau^2\right)\,.
\end{multline*}
In order to lighten the notation, we write $\kappa$ and $\mathscr{N}^\varphi_{\hbar}$ instead of $\kappa_r$ and $\mathscr{N}^\varphi_{\hbar,r}$.  In all what follows we shall use the following notation in order to compare operators and deal with remainders:

\begin{notation} \label{not:granO} For formal operators $A$, $B$, $C$, $\ldots $ in $L^2(\R)$ we say that
$A = \ooo(B, C)$ if there is a constant $c \geq 0$ such that for all $u$ in $\sss(\R)$
\[
\norm{Au} \leq c (\norm{Bu} + \norm{Cu} + \ldots)\,.
\]
This definition naturally extends to $L^2(\R\times \R_+)$ and similar pivot spaces when taking test function satisfying in addition the good boundary conditions.
\end{notation}

\subsection{A pseudo-differential operator with operator-valued symbol}\label{sec.cutoff}

We notice that $\mathscr{N}^\varphi_{\hbar}$ can be written as an $\hbar$-pseudo-differential operator with an operator-valued symbol $n_\hbar(\sigma,\xi)$ having an expansion in powers of $\hbar^{\frac 12}$:
\[\mathscr{N}^\varphi_{\hbar}=\Op n_\hbar\,,\]
with $\Op n_\hbar$ acting on $\mathcal{S}(\R_\sigma, \mathcal{S}(\overline{\R_{+,\tau}}))$ through the usual quantization formula (see \cite[Definition 2.1.7]{Keraval})
\[
\Op n_\hbar\, u (\sigma) = \frac{1}{(2\pi\hbar)} \iint_{\R^2} e^{i(\sigma-\tilde{\sigma})\cdot \xi} n_\hbar \left(\frac{\sigma+\tilde{\sigma}}{2}, \xi\right)u(\tilde{\sigma}) \dd \tilde{\sigma} \dd \xi,
\]
with here
\[n_\hbar=n_0+\hbar^{\frac 12}n_1+\hbar n_2+  \hbar^{\frac 32}n_3 + \hbar^2 \tilde{r}_\hbar  \,,\]
 and where after a computation using the usual symbolic rules, we get
\begin{equation}\label{eq.nj}
\begin{split}
n_0&=-\partial_\tau^2+(\xi-\tau)^2\,,\\
n_1&=2i(\xi-\tau)\varphi'\,,\\
n_2&=-\varphi'^2+\kappa c_\mu \partial_\tau+c_\mu \kappa(\xi-\tau)\tau^2+2\kappa\tau c_\mu (\xi-\tau)^2  +\kappa \tau c_\mu'\left(\tau \right)\,, \\
\Re n_3& = 0\,, \\
 \tilde{r}_{\hbar} &= {\ooo}(\tau^4, (\xi-\tau)^2 \tau^2, (\xi-\tau)\tau, \tau^2\partial_\tau)\,.
\end{split}
\end{equation}
 In the last expression, the notation ${\ooo}$ is defined in Notation \ref{not:granO}. The expansion was performed with respect to $\hbar$, with $\mu$ considered a parameter (see Remark \ref{rem.mu}). It will be explained later how to deal with the remainder $\tilde{r}_\hbar$. It involves in particular powers of $\tau$ which can be controlled via the \emph{normal} localization estimates, and thus are not really problematic.
Note the in \eqref{eq.nj}, $\mu$ is considered as a parameter although it may depend on $\hbar$.

Now the frequency variable $\xi$ is a priori unbounded, and in the next step of the analysis, we therefore \enquote{truncate} our operator in $\xi$ to get a bounded symbol. Let us consider a smooth, bounded, and increasing odd function $\chi$ such that $\chi(\xi)=\xi$ for $\xi\in[-\frac{\xi_0}{2},\frac{\xi_0}{2}]$. We let $\eta_\pm=\pm\lim_{\xi\to\pm\infty}\chi(\xi)$ and assume that $\eta_-\in(0,\xi_0)$.

We let, for all $\xi\in\R$,
\[\chi_1(\xi)=\xi_0+\chi(\xi-\xi_0)\,.\]
Then, the function $\xi\mapsto\mu_1(\chi_1(\xi))$ is bounded  and still has a unique minimum at $\xi_0$, which is non-degenerate and not attained at infinity. Note that, by construction, we have, for all $\xi\in[\frac{\xi_0}{2},\frac{3\xi_0}{2}]$, $\mu_1(\chi_1(\xi))=\mu_1(\xi)$. Since $\mu_1(\xi)<1$ for all $\xi>0$ and $\eta_-\in(0,\xi_0)$, we also notice that
\begin{equation}\label{eq.mu1chi1}
\mu_1\circ\chi_1(\R)\subset[\Theta_0,1)\,.
\end{equation}
We will consider
\begin{equation}\label{eq.p}
\Op p_\hbar\,,\quad \mbox{ with } \quad p_\hbar(s,\xi)=n_\hbar(s,\chi_1(\xi))\,,
\end{equation}
and notice in particular that the principal operator symbol of $\Op p_\hbar$ is
\[p_0(s,\xi)=-\partial^2_\tau+(\chi_1(\xi)-\tau)^2\,.\]
For a recent panorama of pseudo-differential operators with operator symbols, we refer to \cite[Chapitre 2]{Keraval} (see also \cite[Appendix B]{GMS91}). The introduction of the function $\chi_1$ is inspired by \cite[Section 6.3]{Keraval}.

\subsection{The Grushin problem for the principal operator symbol}
Let us first consider the principal symbol $p_0$ (whose domain is independent of $\xi$). Let $z\in\C$ such that $\Re z\in(\Theta_0-\eps,\Theta_0+\eps)$, with $\varepsilon>0$ such that $\Theta_0+\varepsilon<1$. Consider the matrix operator:
\[\mathscr{P}_{0,z}(\xi):=\begin{pmatrix}
p_0-z&\cdot v_\xi\\
\langle\cdot,v_\xi\rangle&0
\end{pmatrix}\in S(\R^2_{s,\xi},\mathscr{L}(\Dom p_0\times\C,L^2(\R_+)\times\C))\,,
\]
acting on $\Dom(p_0)\times\C$ and valued in $L^2(\R_+)\times\C$. Here $v_\xi=u_{\chi_1(\xi)}$. We also denote by $\Pi_\xi$, or simply $\Pi$ the orthogonal projection on $\C v_\xi$.
\begin{notation}\label{not.S}
The notation $P\in S(\R^2,\mathscr{L}(\Dom p_0\times\C,L^2(\R_+)\times\C))$ means that
\begin{enumerate}[---]
\item $P=P(x,\xi)$ is a family of closed operators whose domain does not depend on $(x,\xi)$, and whose graph norms are equivalent uniformly in $(x,\xi)$,
\item for all $\alpha\in\N^2$, there exists $C_\alpha>0$ such that $\|\partial^\alpha_{ s,\xi } P\cdot\|\leq C_\alpha\|\cdot\|_P$, uniformly with respect to $(x,\xi)$, and where $\|\cdot\|_P$ is the graph norm of $P$.
\end{enumerate}
\end{notation}
This class can be thought as a generalization of the standard class of scalar symbols \[S(1)=\{p\in\mathscr{C}^{\infty}(\R^2,\C)\,,\forall\alpha\in\N^2\,,\exists C_\alpha>0\,, \|\partial^\alpha_{s,\xi} p\|\leq C_\alpha\}\]
to operator-valued symbols.
Note however that, contrary to the scalar case, this is not an algebra.  More details can be found in \cite[Section 6.3]{Keraval}.

\begin{lemma}\label{lem.para0}
For all $\xi\in\R$, $\mathscr{P}_{0,z}(\xi)$ is bijective and
\[\mathscr{Q}_{0,z}(\xi):=\mathscr{P}^{-1}_{0,z}(\xi)=\begin{pmatrix}
(p_0-z)^{-1}\Pi^\perp&\cdot v_\xi\\
\langle\cdot,v_\xi\rangle&z-\mu_1(\xi\chi_1(\xi))
\end{pmatrix}\,,
\]
and
\[\mathscr{Q}_{0,z}\in S(\R^2_{s,\xi},\mathscr{L}(L^2(\R_+)\times\C,\Dom p_0\times\C))\,.\]
Here $\Pi^\perp$ denotes the orthogonal projection on ${v_\xi}^\perp$.
\end{lemma}
\begin{proof}
Let $(v,\beta)\in L^2(\R_+)\times\C$ and let us look for $(u,\alpha)\in\Dom(p_0)\times\C$ such that $\mathscr{P}_{0,z}(\xi)(u,\alpha)^T=(v,\beta)^T$. In other words,
\[(p_0-z)u=v-\alpha v_\xi\,,\qquad \langle u,v_\xi\rangle=\beta\,,\]
or
\begin{equation}\label{eq.Fred}
(p_0-z) \Pi^{\perp}u=v-\alpha v_\xi-\beta(p_0-z)v_\xi=v-\alpha v_\xi-\beta( \mu_1(\chi_1(\xi))-z)v_\xi\,,
\end{equation}
with $\langle u,v_\xi\rangle=\beta$.

The operator $p_0-z$ stabilizes $(\C v_\xi)^\perp$ and induces an operator. Moreover, there exists $c>0$ such that for all $u\in \Dom(p_0)\cap (\C v_\xi)^\perp$ and all $z\in\C$ such that $\Re z\in(\Theta_0-\varepsilon,\Theta_0+\varepsilon)$,
\[\Re\langle(p_0-z)u,u\rangle=\langle (p_0-\Re z)u,u\rangle\geq (\mu_2(\chi_1(\xi))-\Re z)\|u\|^2\geq c\|u\|^2\,,\]
where we used the self-adjointness of $p_0$, the min-max principle and the fact that $\min\mu_2>1$ (see \cite[Proposition 3.2.2 \& Remark 3.2.6]{FH10}) and $\Theta_0+\varepsilon<1$.
Thus, the operator $(p_0-z)_{|(\C v_\xi)^\perp}$ is injective with closed range and, by considering the adjoint, it is bijective. We also notice that
\[\|(p_0-z)^{-1}\Pi^\perp\|\leq (\mu_2(\chi_1(\xi))-\Re z)^{-1}\leq c^{-1}\,.\]
The equation \eqref{eq.Fred} has a solution if and only if the r.h.s. belongs to $(\C v_\xi)^\perp$, that is
\[\alpha=\langle v,v_\xi\rangle-\beta(\mu_1(\chi_1(\xi))-z)\,.\]
This unique solution is given by
\[\Pi^{\perp}u=(p_0-z)^{-1}\Pi^\perp(v-\alpha v_\xi-\beta(\mu_1(\chi_1(\xi))-z)v_\xi)=(p_0-z)^{-1}\Pi^\perp v\,.\]
Therefore, $u=\beta v_\xi+(p_0-z)^{-1}\Pi^\perp v$.

\end{proof}

\subsection{Pseudo-differential dimensional reduction and subprincipal terms}
Let us now consider the full symbol
\[\mathscr{P}_{z}(s,\xi):=\begin{pmatrix}
p_\hbar-z&\cdot v_\xi\\
\langle\cdot,v_\xi\rangle&0
\end{pmatrix}\in S(\R^2_{s,\xi}, \mathscr{L}(\Dom p_0\times\C,L^2(\R_+)\times\C))\,,
\]
and notice that we can write
\[\mathscr{P}_{z}=\underbrace{\mathscr{P}_{0,z} + \hbar^{\frac12}\mathscr{P}_{1} + \hbar\mathscr{P}_{2} + \hbar^{\frac32}\mathscr{P}_{3}}_{\mathscr{P}_{z}^{[3]}} + \hbar^2\mathscr{R}_{\hbar}\,,\quad \]
where
$$
\mbox{ for $j\geq 1$}\,,\quad  \mathscr{P}_{j}=\begin{pmatrix}
p_j& 0\\
0&0
\end{pmatrix}\,, \qquad \mathscr{R}_{\hbar}=\begin{pmatrix}
r_\hbar& 0\\
0&0
\end{pmatrix}
$$
and from \eqref{eq.nj} and using the fact that $\chi_1(\xi)$ is now bounded,  we can write
\begin{equation}\label{eq.pj}
\begin{split}
p_0&=-\partial_\tau^2+(\chi_1(\xi)-\tau)^2\,,\\
p_1&=2i(\chi_1(\xi)-\tau)\varphi'\,,\\
p_2&=-\varphi'^2+\kappa c_\mu \partial_\tau+c_\mu \kappa(\chi_1(\xi)-\tau)\tau^2+2\kappa\tau c_\mu (\chi_1(\xi)-\tau)^2  +\kappa \tau c_\mu'\left(\tau \right)\,, \\
\Re p_3& = 0\,, \\
 r_{\hbar} &= {\ooo}(\tau^4, \tau^2\partial_\tau)\,.
\end{split}
\end{equation}

\begin{remark} Note that in the last expansion  at order $3$ w.r.t. $\hbar^{\frac12}$, we do not need the exact expression of $p_3$ and will use later that it is purely imaginary.  The structure of the last Taylor expansion is rather subtle. Indeed we do not care about the cutoff in variable $\tau$ induced by $c_\mu$, but we have to keep in mind that up to loosing powers of $\hbar$, the involved operators are indeed in $S(1)$. This property allows to do all the computations with test functions in $\mathrm{ Dom} (p_0) \times \C$ and gives a meaning to the composition of operators done in the next theorem. In particular,  this expansion is uniform in the parameter $\mu$. Let us notice that the powers of $\tau$ and $\partial_\tau$ in $r_\hbar$ will be compensated later by the normal decay.
\end{remark}

The following theorem gives then  an approximated parametrix of operator $\Op\mathscr{P}_{z}$, that is, in our context, an inverse up to a remainder of order $\hbar^2$.
\begin{theorem}\label{thm.Grushin}
Consider the operator symbol
\[\mathscr{Q}^{[3]}_z=\mathscr{Q}_{0,z}+\hbar^{\frac 12}\mathscr{Q}_{1,z}+\hbar\mathscr{Q}_{2,z}+\hbar^{\frac 32}\mathscr{Q}_{3,z}\,\]	
where $\mathscr{Q}_{0,z}$ is given in Lemma \ref{lem.para0} and
\begin{equation}\label{eq.Qj}\begin{split}
\mathscr{Q}_{1,z}&=-\mathscr{Q}_{0,z}\mathscr{P}_{1}\mathscr{Q}_{0,z}\,,\\
\mathscr{Q}_{2,z}&=-\mathscr{Q}_{0,z}\mathscr{P}_{2}\mathscr{Q}_{0,z}
-\mathscr{Q}_{1,z}\mathscr{P}_{1}\mathscr{Q}_{0,z}\,,\\
\mathscr{Q}_{3,z}&=-\mathscr{Q}_{0,z}\mathscr{P}_{3}\mathscr{Q}_{0,z}
-\mathscr{Q}_{1,z}\mathscr{P}_{2}\mathscr{Q}_{0,z}
-\mathscr{Q}_{2,z}\mathscr{P}_{1}\mathscr{Q}_{0,z}
-\mathscr{C}_z\,,
\end{split}
\end{equation}
 with
\[ 2i\mathscr{C}_z=  \left(\{\mathscr{Q}_{0,z},\mathscr{P}_1\} +
\{\mathscr{Q}_{1,z},\mathscr{P}_{0,z}\} \right) \mathscr{Q}_{0,z}\,,\]
where we used the classical notation for the Poisson bracket
\[\{\mathscr{Q},\mathscr{P}\}=\partial_\xi\mathscr{Q}\cdot\partial_s\mathscr{P}-\partial_s\mathscr{Q}\cdot\partial_\xi\mathscr{P}\,.\]
Then, we have
\begin{equation}\label{eq.approxparam}
\Op(\mathscr{Q}^{[3]}_{z}) \Op(\mathscr{P}_{z}) =\mathrm{Id}
+ \hbar^2 \ooo( \seq{\tau}^6)\,.
\end{equation}
Moreover, we have the following explicit description. Letting
	\[\mathscr{Q}^{[3]}_z=\begin{pmatrix}
	q_z&q^+_z\\
	q^-_z&q_z^\pm
	\end{pmatrix}
	\,,\]
we write
\[q_z^\pm=q_{0,z}^\pm+\hbar^{\frac 12}q_{1,z}^\pm+\hbar q_{2,z}^\pm+\hbar^{\frac 32}q_{3,z}^\pm\,,\]
with
\begin{equation}\label{eq.qj}\begin{split}
q_{0,z}^\pm&=z-\mu_1(\chi_1(\xi))\,,\\
q_{1,z}^\pm&=-i\varphi'(s)\mu_1(\chi_1(\cdot))'(\xi)\,,\\
q_{2,z}^\pm&=\kappa(\sigma)C_1(\xi,\mu)+C_2(\xi,z)\varphi'^2\,,
\end{split}\end{equation}
where
\[\begin{split}
C_1(\xi,\mu)&=\langle \left(c_\mu\partial_\tau+c_\mu (\chi_1(\xi)-\tau)\tau^2+2\tau c_\mu (\chi_1(\xi)-\tau)^2\right)v_\xi,v_\xi\rangle- \langle \tau  c_\mu'(\tau) \partial_\tau v_\xi,v_\xi\rangle\,,\\
C_2(\xi,z)&=1-4\langle(p_0-z)^{-1}\Pi^\perp(\chi_1(\xi)-\tau)v_\xi,(\chi_1(\xi)-\tau)v_\xi\rangle\,.
\end{split}\]	
and  when $z$ is real we have
$$
\Re q_{3,z}^\pm=0\,.
$$
Moreover, $q_z^-$, $q_z^+$, and $q_z^\pm$ are uniformly (with respect to $\mu$) bounded symbols.
\end{theorem}

\begin{remark}
From \cite[Prop. A.2]{FH06} (see also the definition of $C_1$ in \eqref{eq.theta0}), we have
\[C_1(\xi_0,0)=C_1\,,\]
and, from the exponential decay of $v_\xi$  and its derivative (in the $\tau$ variable) and the confinement in $\tau$ induced by the truncation $c_\mu$, we have, uniformly in $\xi$,	
\[C_1(\xi_0,\mu)=C_1+\mathscr{O}(\hbar^\infty)\,,\quad \langle \tau c_\mu'(\tau) \partial_\tau v_\xi,v_\xi\rangle =\mathscr{O}(\hbar^\infty)\,.\]
From \cite[Prop. A.3]{FH06}, we have
\[C_2(\xi_0,\Theta_0)=\frac{\mu''(\xi_0)}{2}\,.\]
\end{remark}
\begin{remark}
Let us recall here that
the bijectivity of  $\Op(p_\hbar)-z$ is related to the one of $\Op(q_z^\pm)$. In this case, we have, modulo some remainders,
\[(\Op(p_\hbar)-z)^{-1}\simeq\Op q_z-\Op q_z^-[\Op q_z^\pm]^{-1}\Op q^+_z\,.\]
\end{remark}

\begin{proof}
The proof is constructive. In order to see where the expressions \eqref{eq.Qj} are coming from, let us consider the product
\[\Op(\mathscr{Q}^{[3]}_{z}) \Op(\mathscr{P}^{[3]}_{z})\,,\]
and its the expansion  in half-powers of $\hbar$. The symbols $\mathscr{Q}_{j,z}$ are chosen so that \eqref{eq.approxparam} holds. Let us explain how these choices are made.

\subsubsection*{\textbf{Terms of order $\hbar^0$}}
The terms of order $1$ give
\[ \mathscr{Q}_{0,z} \mathscr{P}_{0,z}=\mathrm{Id}\,.\]
Now, one wants to cancel the other terms.

\subsubsection*{\textbf{Terms of order $\hbar^{\frac 12}$}}
Cancelling the terms of order $\hbar^{\frac 12}$, we find
\begin{equation}\label{eq.h12} \mathscr{Q}_{1,z}\mathscr{P}_{0,z}+\mathscr{Q}_{0,z}\mathscr{P}_{1}=0\,,\end{equation}
or, equivalently,
\[\mathscr{Q}_{1,z}=-\mathscr{Q}_{0,z}\mathscr{P}_{1,z}\mathscr{Q}_{0,z}\,.\]
Explicitly,
\[\mathscr{Q}_{1,z}=-\begin{pmatrix}
q_{0,z} p_1 q_{0,z}&q_{0,z}p_1 q^+_0\\
q_0^- p_1 q_{0,z}&q^{-}_{0}p_1q^+_0
\end{pmatrix}\,.
\]
Note that
\[q_{1,z}^\pm=-\langle p_1 v_\xi,v_\xi\rangle\,,\quad p_1=2i\varphi'(\chi_1(\xi)-\tau)\,.\]
By the Feynman-Hellmann theorem,
\[q_{1,z}^\pm=-2i\varphi'\langle (\chi_1(\xi)-\tau) v_\xi,v_\xi\rangle=-i\varphi'(s)\mu_1(\chi_1(\cdot))'(\xi)\,.\]

\subsubsection*{\textbf{Terms of order $\hbar^{1}$}}
Let us cancel the terms of order $\hbar$:
\[ \mathscr{Q}_{1,z}\mathscr{P}_{1}+ \frac{1}{2i}\{\mathscr{Q}_{0,z},\mathscr{P}_{0,z}\}
+\mathscr{Q}_{0,z}\mathscr{P}_{2}+\mathscr{Q}_{2,z}\mathscr{P}_{0,z}=0\,.\]
Since the principal symbol does not depend on  $s$,  the Poisson bracket is zero, and thus
\[ \mathscr{Q}_{1,z}\mathscr{P}_{1}+\mathscr{Q}_{0,z}\mathscr{P}_{2}+\mathscr{Q}_{2,z}\mathscr{P}_{0,z}=0\,.\]
It follows that
\[ \mathscr{Q}_{2,z}=-\mathscr{Q}_{1,z}\mathscr{P}_{1}\mathscr{Q}_{0,z}
-\mathscr{Q}_{0,z}\mathscr{P}_{2}\mathscr{Q}_{0,z}\,.\]
We have
\[\mathscr{Q}_{0,z}\mathscr{P}_{2}\mathscr{Q}_{0,z}=\begin{pmatrix}
q_{0,z} p_2 q_{0,z}&q_{0,z}p_2 q^+_0\\
q_0^- p_2 q_{0,z}&\langle p_2 v_\xi,v_\xi\rangle
\end{pmatrix}\,,\]
and   from the expression of $\mathscr{Q}_{1,z}$ above
\[\mathscr{Q}_{1,z}\mathscr{P}_{1}\mathscr{Q}_{0,z}=  -  \begin{pmatrix}
q_0 p_1 q_0 p_1  q_0  &q_0 p_1 q_0 p_1 q_0^+\\
q_0^-p_1 q_0 p_1 q_0&q_0^-p_1q_0p_1q_0^+
\end{pmatrix}\,.\]
In particular, we have
\[q_{2,z}^\pm=q_0^-p_1q_0p_1q_0^+-\langle p_2 v_\xi,v_\xi\rangle=\langle p_1(p_0-z)^{-1}\Pi^\perp p_1 v_\xi,v_\xi\rangle-\langle p_2 v_\xi,v_\xi\rangle\,.\]
With \eqref{eq.nj} and \eqref{eq.p}, we have
\[\begin{split}
\langle p_1(p_0-z)^{-1}\Pi^\perp p_1 v_\xi,v_\xi\rangle&=-4\varphi'^2\langle(p_0-z)^{-1}\Pi^\perp(\chi_1(\xi)-\tau)v_\xi,(\chi_1(\xi)-\tau)v_\xi\rangle\,,\\
\langle p_2 v_\xi,v_\xi\rangle&=-\varphi'^2+\kappa C_1(\xi,\mu)\,.
\end{split}
\]

\subsubsection*{\textbf{Terms of order $\hbar^{\frac 32}$}}
In the same way, we determine $\mathscr{Q}_{3,z}$ by solving
 \[\mathscr{Q}_{0,z}\mathscr{P}_3+\mathscr{Q}_{1,z}\mathscr{P}_2
+\mathscr{Q}_{2,z}\mathscr{P}_1+\mathscr{Q}_{3,z}\mathscr{P}_0
+\frac{1}{2i}\left(\{\mathscr{Q}_{0,z},\mathscr{P}_1\}+\{\mathscr{Q}_{1,z},\mathscr{P}_{0,z}\}\right)=0
\,.\]
which gives
\begin{equation} \label{defQ}
\mathscr{Q}_{3,z} =
-  \mathscr{Q}_{0,z}\mathscr{P}_{3}\mathscr{Q}_{0,z}
+\mathscr{Q}_{1,z}\mathscr{P}_{2}\mathscr{Q}_{0,z}
+\mathscr{Q}_{2,z}\mathscr{P}_{1}\mathscr{Q}_{0,z}  - \mathscr{C}_z
\end{equation}
which is  the last equality in \eqref{eq.Qj}. 

We show now that when $z$ is real,  $\Re (q_3^\pm)$ is purely imaginary.
For this we notice that
the first term  in parenthesis in \eqref{defQ} gives rise to a purely imaginary term in the right bottom of its matrix expression. Then, we show that $\mathscr{C}_z$ is actually skew-self-adjoint. First, since $\mathscr{P}_{0,z}$ does not depend on $s$,
\[
2i\mathscr{C}_z=
 \partial_\xi\mathscr{Q}_{0,z}\partial_s\mathscr{P}_1\mathscr{Q}_{0,z}
- \partial_s\mathscr{Q}_{1,z}\partial_{\xi}\mathscr{P}_{0,z}\mathscr{Q}_{0,z}\,.
\]
Then, recalling that $\mathscr{P}_{0,z}\mathscr{Q}_{0,z}=\mathrm{Id}$ and \eqref{eq.h12} and taking the derivatives of these formulas with respect to $\xi$ and $s$, respectively, we get
\[\begin{split}
2i\mathscr{C}_z&=
\partial_\xi\mathscr{Q}_{0,z}\partial_s\mathscr{P}_1\mathscr{Q}_{0,z}
+\partial_s\mathscr{Q}_{1,z}\mathscr{P}_{0,z}\partial_\xi\mathscr{Q}_{0,z}\\
&=-\partial_\xi\mathscr{Q}_{0,z}\mathscr{P}_{0,z}\partial_s\mathscr{Q}_{1,z}
+\partial_s\mathscr{Q}_{1,z}\mathscr{P}_{0,z}\partial_\xi\mathscr{Q}_{0,z}\\
&=(\mathscr{P}_{0,z}\partial_\xi\mathscr{Q}_{0,z})^*(\partial_s\mathscr{Q}_{1,z})^*
+\partial_s\mathscr{Q}_{1,z}\mathscr{P}_{0,z}\partial_\xi\mathscr{Q}_{0,z}\\
&=(\partial_s\mathscr{Q}_{1,z}\mathscr{P}_{0,z}\partial_\xi\mathscr{Q}_{0,z})^*
+\partial_s\mathscr{Q}_{1,z}\mathscr{P}_{0,z}\partial_\xi\mathscr{Q}_{0,z}\,,
\end{split}\]
where we used that $\mathscr{P}_{0,z}$, $\mathscr{Q}_{0,z}$ are self-adjoint and $\mathscr{P}_1$, $\mathscr{Q}_{1,z}$ are skew-self-adjoint.

\subsubsection*{\textbf{Remainders and order $\hbar^2$.}}

Therefore, with the definition of $\mathscr{Q}^{[3]}_{z}$, and composition of pseudo-differential operators, the operator symbol of $\Op(\mathscr{Q}^{[3]}_{z})\Op(\mathscr{P}^{[3]}_{z})$ coincides with $\mathrm{Id}$ modulo terms of orders at least $\mathscr{O}(\hbar^2)$. By the Calder\'on-Vaillancourt theorem, this remainder is a bounded operator, but the bound depends on the parameter $\mu$. To avoid this problem, we observe that, by Taylor expansion, the remainder is of order $\hbar^2$ in the worse topology of $L^2(\langle\tau\rangle^6\dd\tau \dd s)$. This power $6$ comes from the product of the terms of order $\hbar^{\frac 32}$. In the same way, we see that
\[\Op(\mathscr{Q}^{[3]}_{z})\left(\Op(\mathscr{P}_{z})-\Op(\mathscr{P}^{[3]}_{z})\right)\]
is again of order $\hbar^2$ for the topology $L^2(\langle\tau\rangle^6\dd\tau \dd s)$. Using that
\begin{equation}
\hbar^2 (p_0-z)^{-1} c_\mu \tau^2 \partial_\tau = \hbar^2 \ooo(\seq{\tau}^2),
\end{equation}
we can get rid of the derivatives in the remainder term involving $\tau^2 \partial_\tau$.

The fact that $q_z^-$, $q_z^+$, and $q_z^\pm$ are bounded comes from their explicit expressions and the fact that $v_\xi$ is exponentially decaying uniformly in $\xi$ with respect to $\tau$.

\end{proof}

\section{Tangential coercivity estimates}\label{sec.4}
We will use Theorem \ref{thm.Grushin} for $z\in\C$ such that
\[z=\Theta_0-C_1\kappa_{\max}\hbar+\mathscr{O}(\hbar^2)\,,\]
and assume that $\varphi$ is an appropriate sub-solution of the eikonal equation in the following sense 

\begin{assumption}\label{hyp.phi}
	Let $\varphi\geq 0$ be a Lipschitzian function such that, for all $M>0$ there exist $C,R>0$ such that
	\begin{enumerate}[\rm (i)]
		\item for all $\sigma\in\R$, $\mathfrak{v}(\sigma)-\frac{\mu_1''(\xi_0)}{2}\varphi'(\sigma)^2\geq 0$,
		\item for all $\sigma$ such that $|\sigma-s_r|\geq R\hbar^{\frac 12}$, $\mathfrak{v}(\sigma)-\frac{\mu_1''(\xi_0)}{2}\varphi'(\sigma)^2\geq M \hbar$.
		\end{enumerate}	
\end{assumption}
Note that $\varphi = 0$ is such a subsolution (much more useful solutions will be introduced later) and that for all $\sigma$ such that $|\sigma-s_r|\leq R\hbar^{\frac 12}$, $\mathfrak{v}(\sigma)-\frac{\mu_1''(\xi_0)}{2}\varphi'(\sigma)^2\leq C \hbar$.

\begin{theorem}\label{thm.coercivity}
Let $K>0$. Under Assumption \ref{hyp.phi}, there exist $\hbar_0, c, R_0>0$ such that, for all $R>R_0$, there exists $C_R>0$ such that the following holds. For all $\hbar\in(0,\hbar_0)$ and all $z\in\C$ such that $|z- \Theta_0+C_1\kappa_{\max}\hbar|\leq K\hbar^2$, and for all $\psi\in\Dom(\Op p_\hbar)$,
\[cR^2\hbar^2\|\psi\|\leq \|(\Op p_\hbar-z)\psi\|+C_R\hbar^2\|\chi_0(\hbar^{-\frac 12}R^{-1}(\sigma-s_r)) \psi\|+\hbar^2\|\tau^6\psi\|\,\]
where $\chi_0\in\mathscr{C}^\infty_0(\R)$ is $1$ in a neighborhood of $0$.
\end{theorem}
\begin{remark}
The domain of $\Op p_\hbar$ is
\[\Dom(\Op p_\hbar)=L^2(\R_\sigma, B^2_N(\R_{+,\tau}))\,,\]
with $B^2_N(\R_+)=\{u\in H^2(\R_+) : \tau^2u\in L^2(\R_+)\,, u'(0)=0\}$. In Theorem \ref{thm.coercivity}, we use the convention that if $\tau^6\psi$ does not belong to $L^2(\R^2_+)$, we have $\|\tau^6\psi\|=+\infty$ in which case the inequality is true. The same kind of convention will be used in Section \ref{sec.5}. Anyway, in the proofs, $\psi$ can be assumed to belong to the Schwartz class $\mathcal{S}(\overline{\R^2_+})$. 
\end{remark}

\subsection{From the effective operator...}

\begin{proposition}\label{prop.coercivity-1D}
	Let $K>0$. There exist $h_0, C>0$ such that, for all $z\in\C$ such that $|z- \Theta_0+C_1\kappa_{\max}\hbar|\leq K\hbar^2$,	
	\[\hbar\int_{\R}\left(\mathfrak{v}(\sigma)-\frac{\mu_1''(\xi_0)}{2}\varphi'^2(\sigma)\right)|\psi|^2\dd\sigma-C\hbar^2\|\psi\|^2\leq -\Re\langle \Op q^\pm_z\psi, \psi\rangle\,.\]
	In particular, for some $c>0$ and all $R>0$, there exists $C_R>0$ such that
		\[cR^2\hbar^2\|\psi\|\leq\| \Op q^\pm_z\psi\|+C_R\hbar^2\|\chi_0(\hbar^{-\frac 12}R^{-1}(\sigma-s_r)) \psi\|\,.\]
\end{proposition}
\begin{proof}
Using the assumption on $z$ and \eqref{eq.qj}, we have
\[-\Re q^\pm_z=\mu_1(\chi_1(\xi))-\Theta_0+\hbar\left(-\kappa(\sigma)C_1(\xi,\mu)+C_1\kappa_{\max}-C_2(\xi,\Theta_0)\varphi'^2\right)+\mathscr{O}(\hbar^2)\,,\]	
and also
 \[-\Re q^\pm_z=\mu_1(\chi_1(\xi))-\Theta_0+\hbar\left(-\kappa(\sigma)C_1(\xi,0)+C_1(\xi_0,0)\kappa_{\max}-C_2(\xi,\Theta_0)\varphi'^2\right)+\mathscr{O}(\hbar^2)\,.\]	
We write
 \begin{equation}\label{eq.lb-qpm}
 -\Re q^\pm_z\geq \hbar\left(\mathfrak{v}(\sigma)-C_2(\xi_0,\Theta_0)\varphi'^2(\sigma)\right)+r_\hbar\,,
 \end{equation}
 where
 \[r_\hbar=\mu_1(\chi_1(\xi))-\Theta_0+\hbar s_\hbar\,,\]
 with
 \[|s_\hbar|\leq C\min(1,|\xi-\xi_0|)\,.\]
 Since
 \[\mu_1(\chi_1(\xi))-\Theta_0\geq c\min\left((\xi-\xi_0)^2,1\right)\,,\]
we get, from the Young inequality,
 \begin{equation}\label{eq.y}
 r_\hbar\geq -C\hbar^2\,.
 \end{equation}
 Using \eqref{eq.lb-qpm}, \eqref{eq.y}, and the standard Fefferman-Phong inequality, the result follows.

\end{proof}

\subsection{... to the bidimensional operator}
We can now establish Theorem \ref{thm.coercivity}. Let us recall the relation between $\Op p_\hbar$ and $\Op q_z^\pm$. We have by Theorem \ref{thm.Grushin}
\[
\begin{pmatrix}
\Op q_z&\Op q^+_z\\
\Op q^-_z&\Op q_z^\pm
\end{pmatrix}\begin{pmatrix}
\Op p_\hbar-z&B^*\\
B&0
\end{pmatrix}=\mathrm{Id}
+\mathscr{O}_{L^2(\R\times\R_+,\langle\tau\rangle^6\dd\sigma\dd\tau)\to L^2(\R\times\R_+)}(\hbar^2)\,,\]
where
$ B=\Op (\langle\cdot,v_\xi\rangle)\,$.
In particular,
\begin{equation}\label{eq.gr}
\begin{split}
\Op q_z(\Op p_\hbar-z)+\Op q_z^+ B&=\mathrm{Id}+\mathscr{O}_{L^2(\R\times\R_+,\langle\tau\rangle^6\dd\sigma\dd\tau)\to L^2(\R_+)}(\hbar^2)\,\\
\Op q_z^-(\Op p_\hbar-z)+\Op q_z^\pm B&=\mathscr{O}_{L^2(\R\times\R_+,\langle\tau\rangle^6\dd\sigma\dd\tau)\to L^2(\R\times\R_+)}(\hbar^2)\,.
\end{split}
\end{equation}
Thus,
\[\|\psi\|\leq C\|(\Op p_\hbar-z)\psi\|+C\|B\psi\|+C\hbar^2\|\langle\tau\rangle^6\psi\|\,,\]
and
\[\|\Op q_z^\pm (B\psi)\|\leq C\|(\Op p_\hbar-z)\psi\|+C\hbar^2\|\langle\tau\rangle^6\psi\|\,.\]
From Proposition \ref{prop.coercivity-1D}, we deduce
\[cR^2\hbar^2\|B\psi\|\leq C\|(\Op p_\hbar-z)\psi\|+C_R\hbar^2\|\chi_0(\hbar^{-\frac 12}R^{-1}(\sigma-s_r)) B\psi\|\,,\]
and then, choosing $R$ large enough,
\[\tilde cR^2\hbar^2\|\psi\|\leq C\|(\Op p_\hbar-z)\psi\|+C_R\hbar^2\|\chi_0(\hbar^{-\frac 12}R^{-1}(\sigma-s_r)) B\psi\|+C\hbar^2\|\tau^6\psi\|\,.\]
Moreover, by rescaling and using the fact that the symbol of $B$ only depends on $\xi$, we have $[B,\chi_0(\hbar^{-\frac 12}R^{-1}(\sigma-s_r))]=\mathscr{O}(\hbar^{\frac 12})$, we get
\[cR^2\hbar^2\|\psi\|\leq C\|(\Op p_\hbar-z)\psi\|+\hbar^2\|B\chi_0(\hbar^{-\frac 12}R^{-1}(\sigma-s_r))\psi\|+C\hbar^{\frac 52}\|\psi\|+C\hbar^2\|\tau^6\psi\|\,,\]
and the conclusion follows.

\section{Removing the frequency cutoff}\label{sec.5}
Let us now replace in Theorem \ref{thm.coercivity} the \enquote{truncated} operator $\Op p_\hbar$ (defined in Section \ref{sec.cutoff}) by the operator without frequency cutoff $\mathscr{N}_\hbar^{\varphi}$. This can be done up to convenient additional remainders.
\begin{theorem}\label{thm.coercivity-2D}
	Under Assumption \ref{hyp.phi}, there exist $c, \hbar_0>0$ such that for all $\hbar\in(0,\hbar_0)$ and all $\psi\in\Dom(\mathscr{N}_\hbar^{\varphi})$,
\begin{equation*}
c\hbar^2\|\psi\|\leq \|\langle\tau\rangle^6(\mathscr{N}^{\varphi}_\hbar-z)\psi\|
+\hbar^2\|\chi_0(\hbar^{-\frac 12}R^{-1}(\sigma-s_r)) \psi\|\,,
	\end{equation*}
	and
	\begin{equation*}
	c\hbar^{2}\|\hbar^2 D_\sigma^2\psi\|\leq \|\langle\tau\rangle^6(\mathscr{N}^{\varphi}_\hbar-z)\psi\|
	+\hbar^2\|\chi_0(\hbar^{-\frac 12}R^{-1}(\sigma-s_r)) \psi\|\,.
		\end{equation*}
\end{theorem}

\subsection{Preliminary lemmas}
Let us consider a smooth function $\chi_2=\chi_2(\xi)$ equal to $1$ away from a compact and whose support avoids $\xi_0$.

\begin{lemma}\label{lem.L2H1}
	There exist $C, \hbar_0>0$ such that for all $\hbar\in(0,\hbar_0)$ and all $\psi\in\Dom(\mathscr{N}_\hbar^{\varphi})$,
	\[\|\Op \chi_2\psi\|+\|(\hbar D_\sigma-\tau)\Op\chi_2\psi\|+\|D_\tau\Op\chi_2\psi\|\leq C\|(\mathscr{N}^\varphi_{\hbar}-z)\Op\chi_2\psi\|\,.\]

\end{lemma}

\begin{proof}
	We write
	\begin{multline}\label{eq.L2H1a}
	\Re\langle(\mathscr{N}^\varphi_{\hbar}-z)\Op \chi_2 \psi,\Op\chi_2\psi\rangle\\
	\geq (1+o(1))\langle \left(D_\tau^2+(\hbar D_\sigma-\tau)^2\right)\Op\chi_2\psi,\Op\chi_2\psi\rangle-\Re z\|\Op\chi_2\psi\|^2\,.
	\end{multline}
	Thus, by using the support of $\chi_2$ and the properties of $\mu_1$,
	\[\Re\langle(\mathscr{N}^{\varphi}_\hbar-z)\Op \chi_2 \psi,\Op\chi_2\psi\rangle\geq \left((1+o(1))c_1-\Re z\right)\|\Op\chi_2\psi\|^2\,,\quad c_1>\Theta_0\,.\]
	Using again \eqref{eq.L2H1a} and the Cauchy-Schwarz inequality, the conclusion follows.

\end{proof}

Actually, we have also an \enquote{$H^2$-control} with respect to the \enquote{magnetic derivatives}.
\begin{lemma}\label{lem.H2}
	There exist $C, \hbar_0>0$ such that for all $\hbar\in(0,\hbar_0)$ and all $\psi\in\Dom(\mathscr{N}_\hbar^{\varphi})$,
	\[\|(\hbar D_\sigma-\tau)^2\Op\chi_2\psi\|+\| D_\tau^2\Op\chi_2\psi\|\leq C \|(\mathscr{N}^{\varphi}_\hbar-z)\Op\chi_2\psi\|\,.\]
\end{lemma}

\begin{proof}
This is obtained through standard elliptic estimates by controlling first the magnetic tangential derivative.
\end{proof}

\begin{lemma}\label{lem.H2bis}
	Let $N\in\mathbb{N}$. 	There exist $C, \hbar_0>0$ such that for all $\hbar\in(0,\hbar_0)$ and all $\psi\in\Dom(\mathscr{N}_\hbar^{\varphi})$,
\begin{multline*}
\|\Op\chi_2\psi\|+\|D_\tau\Op\chi_2\psi\|+\|(\hbar D_\sigma-\tau)\Op\chi_2\psi\|
+\| D_\tau^2\Op\chi_2\psi\|\\
+\|(\hbar D_s-\tau)^2\Op\chi_2\psi\|
\leq C\|(\mathscr{N}^{\varphi}_\hbar-z)\psi\|+\mathscr{O}(\hbar^N)\|\psi\|\,.
\end{multline*}
\end{lemma}
\begin{proof}
From Lemmas \ref{lem.L2H1} and \ref{lem.H2}, we have
\begin{multline}\label{eq.0}
\|\Op \chi_2\psi\|+\|(\hbar D_\sigma-\tau)\Op\chi_2\psi\|+\|D_\tau\Op\chi_2\psi\|
\\	+\|(\hbar D_\sigma-\tau)^2\Op\chi_2\psi\|+\| D_\tau^2\Op\chi_2\psi\|
\leq C\|(\mathscr{N}^\varphi_{\hbar}-z)\Op\chi_2\psi\|\,.
\end{multline}
Let us deal with the r.h.s. and notice that
\begin{equation}\label{eq.commutchi2}
\|(\mathscr{N}^{\varphi}_\hbar-z)\Op\chi_2\psi\|\leq \|(\mathscr{N}^{\varphi}_\hbar-z)\psi\|+\|[\mathscr{N}^{\varphi}_\hbar,\Op\chi_2]\psi\|\,.
\end{equation}
Let us consider the commutator. One of the terms is
\begin{equation}\label{eq.1}
\begin{split}
\|[a_\hbar^{-1}D_\tau a_\hbar D_\tau,\Op\chi_2]\psi\|&=\|[ a_\hbar^{-1}\partial_\tau a_\hbar,\Op\chi_2]D_\tau\psi\|\\
&\leq\mathscr{O}(\hbar^\infty)\|D_\tau\psi\|+\hbar\|\Op\underline{\chi_2}D_\tau\psi\|\,,
\end{split}
\end{equation}
where $\underline{\chi_2}$ has a support slightly larger than the one of $\chi_2$, and where we used classical results of composition of pseudo-differential operators. The other term is
\begin{equation}\label{eq.2}
\begin {split}
&\|[a_\hbar^{-1}(\hbar D_\sigma-\tau+\hbar\kappa\tau^2/2) a_\hbar^{-1}(\hbar D_s-\tau+c_\mu\hbar\kappa\tau^2/2),\Op\chi_2]\psi\|\\
&\leq \mathscr{O}(\hbar^\infty)(\|(\hbar D_\sigma-\tau)\psi\|+\|\psi\|)+C\hbar\left(\|\Op\underline{\chi_2}(\hbar D_\sigma-\tau)\psi\|+\|\Op\underline{\chi_2}\psi\|\right)\,.
\end{split}
\end{equation}
Using \eqref{eq.0}, \eqref{eq.commutchi2}, \eqref{eq.1}, and \eqref{eq.2}, an induction argument (on the size of the support of $\chi_2$) provides us with
\begin{multline*}
\|\Op\chi_2\psi\|+\|D_\tau\Op\chi_2\psi\|+\|(\hbar D_\sigma-\tau)\Op\chi_2\psi\|
+\| D_\tau^2\Op\chi_2\psi\|\\
+\|(\hbar D_\sigma-\tau)^2\Op\chi_2\psi\|
\leq C\|(\mathscr{N}^{\varphi}_\hbar-z)\psi\|+\mathscr{O}(\hbar^N)(\|\psi\|+\|(\hbar D_\sigma-\tau)\psi\|+\|D_\tau\psi\|)\,.
\end{multline*}
Noticing that
\[	\|(\hbar D_\sigma-\tau)\psi\|+\|D_\tau\psi\|\leq C\|\mathscr{N}^{\varphi}_\hbar\psi\|+C\|\psi\|\leq C\|(\mathscr{N}^{\varphi}_\hbar-z)\psi\|+C\|\psi\| \,,\]
the conclusion follows.
\end{proof}
\begin{remark}\label{rem.DirNeu}
	The estimates in Lemmas \ref{lem.L2H1}, \ref{lem.H2}, and \ref{lem.H2bis} are also true for $\psi$ satisfying the Dirichlet condition (instead of the Neumann condition).
\end{remark}

We would like to get a control $\hbar D_\sigma$ instead of $\hbar D_\sigma-\tau$. In particular, one should control $\tau$ with the normal Agmon estimates.

\begin{proposition}\label{prop.controlxi2}
		Let $N\in\mathbb{N}$. 	There exist $C, \hbar_0>0$ such that for all $\hbar\in(0,\hbar_0)$ and all $\psi\in\Dom(\mathscr{N}_\hbar^{\varphi})$,
	\begin{multline*}
\|\Op\chi_2\psi\|+\|D_\tau\Op\chi_2\psi\|+\|\hbar D_s\Op\chi_2\psi\|+\| D_\tau^2\Op\chi_2\psi\|+\|(\hbar D_s)^2\Op\chi_2\psi\|
\\
+\|\tau \hbar D_s \Op \chi_2\psi\|\leq C\|(\mathscr{N}^{\varphi}_\hbar-z)\psi\|+\|\tau(\mathscr{N}^{\varphi}_\hbar-z)\psi\|+\|\tau^2(\mathscr{N}^{\varphi}_\hbar-z)\psi\|+\mathscr{O}(\hbar^N)\|\psi\|\,.
\end{multline*}	
\end{proposition}

\begin{proof}
Let us apply Lemma \ref{lem.H2bis} to $\tau\psi$ (recall Remark \ref{rem.DirNeu}). We get
\[\|\tau\Op\chi_2\psi\|\leq C\|(\mathscr{N}^{\varphi}_\hbar-z)\tau\psi\|+\mathscr{O}(\hbar^N)\|\tau\psi\|\,.\]	
Replacing $\psi$ by $\Op\underline{\chi_2}\psi$, commuting $\mathscr{N}^{\varphi}_\hbar$ with $\tau$ and using Lemma \ref{lem.H2bis}, we get
\[\|\tau\Op\chi_2\psi\|\leq C\|\tau(\mathscr{N}^{\varphi}_\hbar-z)\psi\|+\mathscr{O}(\hbar^N)(\|\psi\|+\|\tau\psi\|)\,.\]
With Lemma \ref{lem.H2bis}, we get
\begin{equation*}
\|(\hbar D_s)\Op\chi_2\psi\|
\leq C\|(\mathscr{N}^{\varphi}_\hbar-z)\psi\|+ C\|\tau(\mathscr{N}^{\varphi}_\hbar-z)\psi\|+\mathscr{O}(\hbar^N)(\|\psi\|+\|\tau\psi\|)\,.
\end{equation*}
In the same spirit, we get
\begin{multline*}
\|(\hbar D_s)^2\Op\chi_2\psi\|
\leq C\|(\mathscr{N}^{\varphi}_\hbar-z)\psi\|+ C\|\tau(\mathscr{N}^{\varphi}_\hbar-z)\psi\|+C\|\tau^2(\mathscr{N}^{\varphi}_\hbar-z)\psi\|\\
+\mathscr{O}(\hbar^N)(\|(1+\tau+\tau^2)\psi\|)\,,
\end{multline*}
\begin{multline*}
\|\tau\hbar D_s\Op\chi_2\psi\|
\leq C\|(\mathscr{N}^{\varphi}_\hbar-z)\psi\|+ C\|\tau(\mathscr{N}^{\varphi}_\hbar-z)\psi\|+C\|\tau^2(\mathscr{N}^{\varphi}_\hbar-z)\psi\|\\
+\mathscr{O}(\hbar^N)(\|(1+\tau+\tau^2)\psi\|)\,.
\end{multline*}
Due to the Dirichlet condition, we have, for $k\geq 1$,
\begin{equation}\label{eq.tauk}
\|\tau^k\psi\|\leq C \|(\mathscr{N}^{\varphi}_\hbar-z)\tau^k\psi\|\,.
\end{equation}
Computing commutators and controlling them by $\|\mathscr{N}^{\varphi}_\hbar\psi\|$, the result follows upon noticing that
\[\|\mathscr{N}^{\varphi}_\hbar\psi\|\leq \|(\mathscr{N}^{\varphi}_\hbar-z)\psi\|+|z|\|\psi\|\,.\]
\end{proof}

\subsection{Proof of Theorem \ref{thm.coercivity-2D}}

With the triangle inequality,
\begin{equation}\label{eq.triangle}
\|(\mathscr{N}^{\varphi}_\hbar-z)\psi\|\geq\|(\mathscr{P}_\hbar-z)\psi\|-\|(\mathscr{N}^{\varphi}_\hbar-\mathscr{P}_\hbar)\psi\|\,.
\end{equation}
Using Proposition \ref{prop.controlxi2} with $\chi_2$ such that $1-\chi_2$ is supported in $\{\chi_1(\xi)=\xi\}$ to control the terms $(\hbar D_s)^2$ and $\hbar\tau D_\sigma$, we get
\begin{multline}\label{eq.xi2}
\|(\mathscr{N}^{\varphi}_\hbar-\mathscr{P}_\hbar)\psi\|
\leq C\|(\mathscr{N}^{\varphi}_\hbar-z)\psi\|+C\|\tau(\mathscr{N}^{\varphi}_\hbar-z)\psi\|+C\|\tau^2(\mathscr{N}^{\varphi}_\hbar-z)\psi\|\\
+\mathscr{O}(\hbar^N)\|\psi\|\,.
\end{multline}
Combining \eqref{eq.triangle} and \eqref{eq.xi2} with Theorem \ref{thm.coercivity}, provides us with
\begin{equation}\label{eq.tau6tobecontrolled}
cR^2\hbar^2\|\psi\|\leq \|\langle\tau\rangle^2(\mathscr{N}^{\varphi}_\hbar-z)\psi\|+C_R\hbar^2\|\chi_0(\hbar^{-\frac 12}R^{-1}(\sigma-s_r)) \psi\|+\hbar^2\|\tau^6\psi\|+C\hbar^N\|\psi\|\,.
\end{equation}
By using again \eqref{eq.tauk}, we get
\[\|\tau^6\psi\|\leq C\|(\mathscr{N}^{\varphi}_\hbar-z)\tau^6\psi\|\leq C\|\tau^6(\mathscr{N}^{\varphi}_\hbar-z)\psi\|+C\|[\mathscr{N}^{\varphi}_\hbar,\tau^6]\psi\|\,.\]
Computing explicitly the commutator, we get, by induction,
\[
\begin{split}
\|\tau^6\psi\|&\leq C\|\langle\tau\rangle^6(\mathscr{N}^{\varphi}_\hbar-z)\psi\|+C\|\mathscr{N}^{\varphi}_\hbar\psi\|\\
&\leq C\|\langle\tau\rangle^6(\mathscr{N}^{\varphi}_\hbar-z)\psi\|+C\|(\mathscr{N}^{\varphi}_\hbar-z)\psi\|+C|z|\|\psi\|\,.
\end{split}\]
With \eqref{eq.tau6tobecontrolled} and choosing $R$ large enough (to absorb the $C|z|$ term), we deduce the first estimate in Theorem \ref{thm.coercivity-2D}.
Combining this estimate with Proposition \ref{prop.controlxi2}, the conclusion follows.

\subsection{A slight improvement}
The considerations in the previous section give the following improvement of Theorem \ref{thm.coercivity-2D}.

\begin{corollary}\label{cor.tau}
	Under Assumption \ref{hyp.phi}, there exist $c, \hbar_0>0$ such that for all $\hbar\in(0,\hbar_0)$ and all $\psi\in\Dom(\mathscr{N}_\hbar^{\varphi})$,
\begin{equation}\label{eq.tauN}
c\hbar^2\|\langle\tau\rangle\psi\|\leq \|\langle\tau\rangle^6(\mathscr{N}^{\varphi}_\hbar-z)\psi\|
+\hbar^2\|\chi_0(\hbar^{-\frac 12}R^{-1}(\sigma-s_r)) \psi\|\,,
\end{equation}	
	and
\begin{equation}\label{eq.tauN2}
c\hbar^{2}\|\langle\tau\rangle\hbar^2 D_\sigma^2\psi\|\leq \|\langle\tau\rangle^6(\mathscr{N}^{\varphi}_\hbar-z)\psi\|
+\hbar^2\|\chi_0(\hbar^{-\frac 12}R^{-1}(\sigma-s_r)) \psi\|\,.
\end{equation}
\end{corollary}
\begin{proof}
We recall \eqref{eq.tauk}, and we use it with $k=1$. Estimating a commutator, this shows that
\[\|\tau\psi\|\leq C\|\tau(\mathscr{N}^{\varphi}_\hbar-z)\psi\|+C\|\partial_\tau\psi\|+C\|\psi\|\,.\]
We also notice that
\begin{equation}\label{eq.dtauNpsipsi}
\|\partial_\tau\psi\|^2\leq C\|\mathscr{N}^{\varphi}_\hbar\psi\|\|\psi\|\,.
\end{equation}
Then, since $z$ is bounded,
\begin{equation}\label{eq.taudtaupsi}
\|\partial_\tau\psi\|+\|\tau\psi\|\leq C\|\langle\tau\rangle(\mathscr{N}^{\varphi}_\hbar-z)\psi\|+C\|\psi\|\,.
\end{equation}
Applying Theorem \ref{thm.coercivity-2D}, we get \eqref{eq.tauN}.

To get \eqref{eq.tauN2}, we apply Proposition \ref{prop.controlxi2} with $\psi$ replaced by $\tau\psi$. Then, we estimate the commutators by using \eqref{eq.taudtaupsi}, \eqref{eq.dtauNpsipsi} (with $\psi$ replaced by $\tau^k\psi$, $k=1,2$) and \eqref{eq.tauk} (with $k=2$), and we use \eqref{eq.tauN}.
\end{proof}

\section{Optimal tangential Agmon estimates}\label{sec.6}

\subsection{Agmon estimates}
Let us discuss here some important consequences of our elliptic estimates. An immediate corollary of Theorem \ref{thm.coercivity-2D} is the following.
\begin{corollary}\label{cor.Agmons}
Under Assumption \ref{hyp.phi} and with the notation introduced in Section~\ref{sec.defonewell}, for all $K>0$, there exist $C, \hbar_0>0$ such that for all $\hbar\in(0,\hbar_0)$ and all $\lambda$ eigenvalue of $\mathscr{N}_
	{\hbar,r}$ such that $\left|\lambda-\left(\Theta_0-C_1\kappa_{\max}\hbar\right)\right|\leq K\hbar^2$ and all associated eigenfunction $\Psi\in\Dom(\mathscr{N}_{\hbar,r})$,
	\[\int_{\R^2_+}e^{2\varphi/\hbar^{\frac 12}}|\Psi|^2\dd s\dd\tau\leq C\|\Psi\|^2\,.\]
\end{corollary}
\begin{proof}
	We apply Theorem \ref{thm.coercivity-2D} with $z=\lambda$ and $\psi=e^{\varphi/\hbar^{\frac 12}}\Psi$.	
\end{proof}
Let us now explain how to get tangential Agmon estimates for the two wells operator $\mathscr{N}_{\hbar}$ from the estimates on the one well operators (acting on $L^2(\R\times\R_+)$). Let us recall the two Agmon distances
\[\begin{split}
\Phi_r(\sigma)&=\sqrt{\frac{2C_1}{\mu''_1(\xi_0)}}\int_{[s_r,\sigma]}\sqrt{\kappa_{\max}-\kappa_r(\tilde\sigma)}\dd\tilde\sigma\,,\\
\Phi_\ell(\sigma)&=\sqrt{\frac{2C_1}{\mu''_1(\xi_0)}}\int_{[s_\ell,\sigma]}\sqrt{\kappa_{\max}-\kappa_\ell(\tilde\sigma)}\dd\tilde\sigma\,.
\end{split}
\]
From them, we can construct a weight to estimate the decay of the eigenfunctions of $\mathscr{N}_{\hbar}$ away from $s_r$ and $s_\ell$. Let us consider some periodic versions of $\Phi_r$ and $\Phi_\ell$. We let
\begin{equation*}
\tilde\Phi_r(\sigma) = \left\{
\begin{array}{ll}
\Phi_r(\sigma)&\mbox{if }-L\leq \sigma\leq s_\ell-\eta
\\
\Phi_r(\sigma-2L)&\mbox{if } s_\ell+\eta<\sigma< L
\end{array}
\right.\,,
\end{equation*}
\begin{equation*}
\tilde\Phi_\ell(\sigma) = \left\{
\begin{array}{ll}
\Phi_\ell(\sigma+2L)&\mbox{if }-L\leq \sigma\leq s_r-\eta
\\
\Phi_\ell(\sigma)&\mbox{if } s_r+\eta<\sigma< L
\end{array}
\right.\,.
\end{equation*}
The functions $\tilde\Phi_r$ and $\tilde\Phi_\ell$ are defined on $[-L,L)$ but not on $(s_\ell-\eta,s_\ell+\eta]$ and $(s_r-\eta,s_r+\eta]$, respectively. Thus, we consider smooth extensions of $\tilde\Phi_r$ and $\tilde\Phi_\ell$ such that $\tilde\Phi_r>\tilde\Phi_\ell$ near $s_\ell$ and $\tilde\Phi_\ell>\tilde\Phi_r$ near $s_r$. Thus, $\tilde\Phi_r$ and $\tilde\Phi_\ell$ can be seen as $2L$-periodic functions.

For the following we shall identify functions on  $\Gamma$ with $2L$-periodic functions and mainly consider $[-L,L)$ as interval of integration. The preceding construction of weights on $[-L,L)$ is adapted to this point of view and will allow to give estimates on the two-well case in this setting.

\begin{proposition}\label{prop.Agmond}
Set $\theta\in(0,1)$ and consider the $2L$-periodic function defined on $[-L,L)$ by
\[\varphi=\sqrt{1-\theta}\min(\tilde\Phi_r,\tilde\Phi_\ell)\,.\]
Let $\varepsilon>0$ and assume that $\eta$ is small enough. There exist $C,\hbar_0>0$ such that for all $\hbar\in(0,\hbar_0)$ and all $\lambda$ eigenvalue of $\mathscr{N}_
{\hbar}$ such that $\left|\lambda-\left(\Theta_0-C_1\kappa_{\max}\hbar\right)\right|\leq K\hbar^2$ and all associated eigenfunction $u\in\Dom(\mathscr{N}_{\hbar})$,
\[\int_{[-L,L)\times\R_+}e^{2\varphi/\hbar^{\frac 12}}|u|^2\dd s\dd\tau\leq Ce^{\varepsilon/\hbar^{\frac 12}}\|u\|^2_{L^2([-L,L)\times\R^+)}\,.\]
\end{proposition}
\begin{proof}
Consider an eigenfunction $u$ as in the assumptions. Let $\chi_{r}$ be a smooth cutoff function equal to $1$ near $s_r$ and being $0$ near $s_\ell$. Away from the support of $\chi_r$, $\varphi$ can be modified and extended to $\R\times\R^+$ so that Assumption \ref{hyp.phi} is satisfied. We can then consider $\psi=\chi_{r}e^{i\sigma\gamma_0/\hbar^2}e^{\varphi/\hbar^{\frac 12}}u$ as a function on $\R$ and apply to it Theorem \ref{thm.coercivity-2D} with $z=\lambda$. We get then
\begin{multline*}
c\hbar^2\|\chi_r e^{\varphi/\hbar^{\frac 12}}u\|\leq
\|e^{\varphi/\hbar^{1/2}}\langle\tau\rangle^6(\mathscr{N}_\hbar-\lambda)(\chi_r u)\|
+\hbar^2\|\chi_0(\hbar^{-\frac 12}R^{-1}(\sigma-s_r)) e^{\varphi/\hbar^{\frac 12}}u\|\,.
\end{multline*}
By using that $u$ is an eigenfunction and by choosing $\eta$ small enough (and adapting $\chi_r$ accordingly), we get the following estimate,
\begin{equation*}
c\hbar^2\|\chi_r e^{\varphi/\hbar^{\frac 12}} u\|\leq e^{\varepsilon/2\hbar^{\frac 12}}\|\langle\tau\rangle^6[\mathscr{N}_\hbar,\chi_r]u\|_{L^2([-L,L)\times\R^+)}
+C\hbar^2\|u\|_{L^2([-L,L)\times\R^+)}\,.
\end{equation*}
Thanks to the normal Agmon estimates, we get
\begin{equation*}
\|\chi_r e^{\varphi/\hbar^{\frac 12}} u\|\leq Ce^{\varepsilon/\hbar^{\frac 12}}\|u\|_{L^2([-L,L)\times\R^+)}\,.
\end{equation*}
By considering the left well, we get by symmetry
\begin{equation*}
\|\chi_\ell e^{\varphi/\hbar^{\frac 12}} u\|\leq Ce^{\varepsilon/\hbar^{\frac 12}}\|u\|_{L^2([-L,L)\times\R^+)}\,.
\end{equation*}
Since the supports of $\chi_r$ and $\chi_\ell$ overlap $[-L,L)$, the conclusion follows.
\end{proof}

\subsection{WKB approximation in the right well}\label{sec.apprx1well}
Let us now discuss a crucial application of Theorem \ref{thm.coercivity-2D}. We work here on the real line. Let us apply the theorem to
\begin{equation}\label{eq:u}
\psi=e^{\varphi/\hbar^{\frac 12}}(\psi_{\hbar,r}-\Pi_{r}\psi_{\hbar,r})\,,
\end{equation}
where
\begin{enumerate}[\rm ---]
	\item $\psi_{\hbar,r}(\sigma,\tau)=\chi_{\eta,r}\Psi_{\hbar,\tau}(\sigma,\tau)$,
	\item $\chi_{\eta,r}$ is a cut-off function supported in $I_{\eta,r}$ and such that $\chi_\eta=1$ on $I_{2\eta,r}$,
	\item $\Psi_{\hbar,r}$ is the WKB solution introduced in \eqref{eq.psir} (with $n=1$) and scaled so that $\|\psi_{\hbar,r}\|=1$ (see Remark \ref{rem.tildef}),
	\item $\Pi_r$ is the orthogonal projection on the first eigenspace, spanned by $u_{\hbar,r}$, of the operator $\mathscr{N}_{\hbar,r}$.
\end{enumerate}
Theorem~\ref{thm.coercivity-2D} and Corollary \ref{cor.tau} yield the following WKB approximations (see, for instance, \cite[Prop. 5.1]{HKR17} in the context of the Robin Laplacian for a similar estimate).
\begin{proposition}\label{prop:WKB-app}
We have
	\begin{equation}\label{eq.appWKB0}
	\|\psi_{\hbar,r}-\Pi_r\psi_{\hbar,r}\|_{L^2(\R^2_+)}=\mathscr{O}(\hbar^\infty)\,,
	\end{equation}
	and we can assume that $\langle\psi_{\hbar,r},u_{\hbar,r}\rangle=1+\mathscr{O}(\hbar^\infty)$ up to the multiplication of $\Psi_{\hbar,r}$ by a complex number of modulus $1$.
	
Moreover, let $K\subset I_{2\eta,r}$ be a compact set. The following estimate
	\begin{equation}\label{eq:app.WKB1'}
	\langle\tau\rangle e^{\Phi_{r}/\sqrt{\hbar}}(\Psi_{\hbar,r}-u_{\hbar,r})=\mathscr{O}(\hbar^\infty)\,,
	\end{equation}
	holds in $\mathscr{C}^1(K;L^2(\R_+))$.
\end{proposition}

\begin{remark} Note that the choice $\langle\psi_{\hbar,r},u_{\hbar,r}\rangle=1+\mathscr{O}(\hbar^\infty)$, in addition to the normalization of $\psi_{\hbar,r}$,  completely determines the quasimode, especially  the value of $\alpha_{1,0}$ which was only defined up to an additive constant in \eqref{eq.alphanprime}.
\end{remark}

\begin{proof}
To get \eqref{eq.appWKB0}, we remember that the expansion of the first eigenvalue is given by $\delta_1(\hbar)$ in Theorem \ref{thm.BKW}, we use the spectral gap of the one well case (see for instance \eqref{eq.HM}) and we apply the spectral theorem.
	
To get \eqref{eq:app.WKB1'}, we use our normalization of the WKB Ansatz, and we choose
	\begin{equation}\label{eq.poids3}
	\varphi(s)=\hat{\Phi}_{r, \eta, N,\hbar} (s) = \min \left\{\tilde{\Phi}_{r, N, \hbar} (s), \sqrt{1-\theta }\,\displaystyle{\inf_{\sigma \in I_{2\eta,r}\setminus I_{\eta,r} } \left(\Phi_{r}(\sigma) +\int_{[s,\sigma]}\sqrt{ V (\tilde\sigma)}\, \dd\tilde\sigma\right)}\right\},
	\end{equation}
	where
	\begin{equation}\label{eq.poids2}
	\tilde{\Phi}_{r, N, \varepsilon} (s)=\Phi_{r}(s)-N\sqrt{\hbar}\,\ln\max\left(\frac{\Phi_{r}}{\sqrt\hbar},N\right)\,.
	\end{equation}
	Here $N\in\N$, $0<\theta<1$.
	\end{proof}

\section{Interaction matrix and tunneling effect}\label{sec.7}
We now have all the elements in hand to prove Theorem \ref{thm.tunnel}. Let us consider the common \enquote{single well} groundstate energy $\mu_1^{\rm sw}(\hbar)$ of the operators $\mathscr{N}_{\hbar,r}$ and $\mathscr{N}_{\hbar,\ell}$ (it depends on $\eta$). It results from the Agmon estimates in Corollary \ref{cor.Agmons} and Proposition \ref{prop.Agmond}, and the min-max principle that
\begin{equation}\label{eq:sp-rough}
\mu_1^{\rm sw}(\hbar)-\tilde{\mathscr O}(e^{-\mathsf{S}/\sqrt{\hbar} })\leq \nu_1(\hbar)\leq \nu_2(\hbar)\leq \mu_1^{\rm sw}(\hbar)+\tilde{\mathscr O}(e^{-\mathsf{S}/\sqrt{\hbar}})\,,
\end{equation}
where $\tilde{\mathscr O}(e^{-\mathsf{S}/\sqrt{\hbar}})$  means ${\mathscr O}(e^{-(\mathsf{S}-\varepsilon)/\sqrt{\hbar}})$ for all $\varepsilon>0$.

\subsection{WKB quasimodes and approximated basis}\label{sec.redinteract}
In this section, we recall the main lines of the strategy to reduce the asymptotic study of the spectral gap $\nu_2(\hbar)-\nu_1(\hbar)$ to the study of the two by two interaction matrix. \emph{Once the tangential exponential decay of the eigenfunctions is established} (see Proposition \ref{prop.Agmond}), the derivation of the interaction matrix can be done as if we were in dimension one. In this section, one will precisely refer to the estimates obtained in \cite[Section 3]{BHR17} where the strategy has been described in great detail for an electric Hamiltonian in dimension one (see also the Bourbaki exposé \cite[Section 2]{Robert} describing the Helffer-Sjöstrand results in \cite{HS84}).

To construct the interaction matrix, we will use the ground states of the one well problems and use them to provide an approximate basis of the space
\[E=\bigoplus_{i=1}^2\,{\rm Ker}(\mathscr{N}_\hbar-\nu_i(\hbar))\,.\]
We will truncate them, project them on $E$ and orthonormalize them.

\subsubsection{Truncation}
Let $\chi_{\eta,r}$ (respectively $\chi_{\eta,\ell}$) be a cut-off function satisfying
$\chi_{\eta,r}=1$ in $\{|s-s_{\ell}|\geq 2\eta\}$ (respectively $\chi_{\eta,\ell}=1$ in $\{|s-s_r| \geq 2\eta\}$) and
$\chi_{\eta,r}=0$ in $\{|s-s_{\ell}|\leq \eta\}$ (respectively $\chi_{\eta,\ell}=0$ in $\{|s-s_r| \leq \eta\}$).

We define, for $\alpha\in\{\ell,r\}$,
\begin{equation}\label{eq:int-qm}
f_{\hbar,\alpha}=\chi_{\eta,\alpha}\phi_{\hbar,\alpha}\,,
\end{equation}
where the $\phi_{\hbar,\alpha}$ are essentially the functions $\check\phi_{\hbar,\alpha}$ (see \eqref{eq.phir0} and \eqref{eq.phil0}) seen on the circle identified with $[-L,L)$, and precisely defined by
\begin{equation}\label{eq.phir}
\phi_{\hbar, r}(\sigma,\tau) = \left\{
\begin{array}{ll}
e^{ -i\gamma_0 \sigma/\hbar^2}u_{\hbar,r}(\sigma,\tau)&\mbox{if }-L\leq \sigma\leq s_\ell-\frac{\eta}{2}
\\
e^{ -i\gamma_0 (\sigma-2L)/\hbar^2}u_{\hbar,r}(\sigma-2L,\tau)&\mbox{if } s_\ell+\frac{\eta}{2}<\sigma< L
\end{array}
\right.\,,
\end{equation}
\begin{equation}\label{eq.phil}
\phi_{\hbar, \ell}(\sigma,\tau) = \left\{
\begin{array}{ll}
e^{-i\gamma_0 (\sigma+2L)/\hbar^2}u_{\hbar,\ell}(\sigma+2L,\tau)&\mbox{if }-L\leq \sigma\leq s_r-\eta/2
\\
e^{-i\gamma_0 \sigma/\hbar^2}u_{\hbar,\ell}(\sigma,\tau)&\mbox{if } s_r+\eta/2<\sigma< L
\end{array}
\right.\,.
\end{equation}
Note here that, due to the flux term $\gamma_0 \sigma/\hbar^2$, there is no  natural extension by periodicity. In the following, we work on $\Gamma$ identified with $[-L,L)$.

Thanks to Proposition \ref{prop.Agmond}, the set $\{f_{\hbar,\ell},f_{\hbar,r}\}$ is quasi-orthonormal in the sense that
\[\|f_{\hbar,\alpha}\|^2=1+\tilde{\mathscr O}(e^{-2\mathsf{S}/\sqrt{\hbar}})\quad{\rm and}\quad
\langle f_{\hbar,\alpha},f_{\hbar,\beta}\rangle=\tilde{\mathscr O}(e^{-\mathsf{S}/\sqrt{\hbar}})~{\rm for~}\alpha\not=\beta\,.\]
Furthermore, the function $r_{\hbar,\alpha}=(\mathscr{N}_{\hbar}-\mu^{\rm sw}(\hbar))f_{\hbar,\alpha}$, $\alpha\in\{\ell,r\}$, satisfies,
\[\|r_{\hbar,\alpha}\|=\tilde{\mathscr O}(e^{-\mathsf{S}/\sqrt{\hbar}})\,.\]
These estimates, in dimension one, are proved, for instance, in \cite[Lemma 3.5]{BHR17}.
\subsubsection{Projection}
Since we want to describe the first two eigenvalues of $\mathscr{N}_\hbar$, it is convenient to build a basis of $E$ from the quasimodes $f_{\hbar,r}$ and $f_{\hbar,\ell}$.
Thus, we consider the new quasimodes, for $\alpha\in\{\ell,r\}$,
\begin{equation}\label{eq:int-qm*}
g_{\hbar,\alpha}=\Pi f_{\hbar,\alpha}\,,
\end{equation}
where $\Pi$ is the orthogonal projection on $E$. The following estimate holds, for $\alpha\in\{\ell,r\}$,
\[\|g_{\hbar,\alpha}-f_{\hbar,\alpha}\|+\left\|\partial_s\left(g_{\hbar,\alpha}-f_{\hbar,\alpha}\right)\right\|
=\tilde{\mathscr O}(e^{-\mathsf{S}/\sqrt{\hbar}})\,,\]
and its proof is the same as the one of \cite[Lemma 3.8]{BHR17}.
\subsubsection{Orthonormalization}
Starting from the basis $\{g_{\hbar,\ell},g_{\hbar,r}\}$, we obtain by the Gram-Schmidt algorithm the orthonormal basis $\{\tilde g_{\hbar,\ell},\tilde g_{\hbar,r}\}$. In other words, we have $\tilde g=g\mathsf{G}^{-\frac12}$ where $\mathsf{G}$ is the Gram-Schmidt matrix $(\langle g_{\hbar,\alpha},g_{\hbar,\beta}\rangle)_{\alpha,\beta\in\{r,\ell\}}$.

Let $\mathsf M$ be the matrix of $\mathscr{N}_\hbar$ in the basis $\{\tilde g_{\hbar,\ell},\tilde g_{\hbar,r}\}$. We have
\[\mathrm{Spec}(\mathsf M)=\{\nu_1(\hbar),\nu_2(\hbar)\}\]
and, by solving the equation ${\rm det}(\mathsf M-\lambda{\mathrm Id})=0$, we deduce, as in dimension one (see \cite[Proposition 3.11]{BHR17}), that
\begin{equation}\label{eq.tunnel0}
\nu_2(\hbar)-\nu_1(\hbar)=2|w_{\ell,r}| +\tilde{\mathscr O}(e^{-2\mathsf{S}/\sqrt{\hbar}})\,,\quad w_{\ell,r}=\langle r_{\hbar,\ell},f_{\hbar,r}\rangle\,.
\end{equation}

\subsection{Computing the interaction}
We may now estimate the interaction term $w_{\ell,r}$. In contrast with Section \ref{sec.redinteract}, we provide here more details since the proof deviates from the usual computation of the interaction in dimension one. Firstly, the tangential derivative $\hbar D_\sigma$ is replaced by the magnetic derivative
\begin{equation}\label{eq.Dh}
\mathscr{D}_\hbar=\hbar D_\sigma+\hbar^{-1}\gamma_0-\tau+\hbar c_\mu \frac{\kappa}{2}\tau^2\,.
\end{equation}
Lemma \ref{lem.ultime} gives an explicit formula for $w_{\ell,r}$ involving this magnetic derivative. Secondly, the separation of variables, responsible for the final reduction to an interaction in dimension one, is explained in Section \ref{sec.fin}.
\subsubsection{An explicit formula for $w_{\ell,r}$}
The aim of this section is to prove the following.
\begin{lemma}\label{lem.ultime}
We have
\begin{multline}\label{eq.ultime}
w_{\ell,r}=
i\hbar\int_{0}^{+\infty}a_\hbar^{-1}\left(\phi_\ell\overline{\mathscr{D}_\hbar\phi_r}+\mathscr{D}_\hbar\phi_\ell\overline{\phi_r}\right)(0,\tau)\\
-a_\hbar^{-1}\left(\phi_\ell\overline{\mathscr{D}_\hbar\phi_r}+\mathscr{D}_\hbar\phi_\ell\overline{\phi_r}\right)(-L,\tau)\dd\tau\,.
\end{multline}	
\end{lemma}

\begin{proof}
 We have
\[w_{\ell,r}=\langle(\mathscr{N}_\hbar-\mu_1^{\rm sw}(\hbar))f_{\hbar,\ell},f_{\hbar,r} \rangle=\langle [\mathscr{N}_\hbar,\chi_{\eta,\ell}]\phi_{\hbar,\ell},\chi_{\eta,r} \phi_{\hbar,r}\rangle\,. \]
For shortness we use in the following the notation $\chi_{r}$ for $\chi_{\eta,r}$ and similarly on the left-side. We recall that  $\chi_{r}$ and $\chi_{\ell}$ do not depend on $\tau$. Thus
\begin{equation}\label{eq.wlr}
w_{\ell,r}=\langle[a_{\hbar}^{-1}\mathscr{D}_\hbar a_\hbar^{-1}\mathscr{D}_\hbar,\chi_\ell]\phi_{\hbar,\ell},\chi_r\phi_{\hbar,r}\rangle\,.
\end{equation}
For shortness, we let $\phi_{\hbar,\alpha}=\phi_{\alpha}$.

In the following we let $S_L=(-L,L)\times(0,+\infty)$. Writing the commutator, integrating by parts, and using the Leibniz formula, we get
\begin{equation*}
\begin{split}
w_{\ell, r}&=\int_{S_L}\left( \mathscr{D}_\hbar(a_\hbar^{-1}\mathscr{D}_\hbar(\chi_\ell\phi_\ell))\chi_r\overline{\phi_r}-\chi_\ell\chi_r\overline{\phi_r}\mathscr{D}_\hbar(a_\hbar^{-1}\mathscr{D}_\hbar \phi_\ell)\right) \dd\sigma\dd\tau\\
&=\int_{S_L} a_\hbar^{-1}\left(\mathscr{D}_\hbar(\chi_\ell\phi_\ell)\overline{\mathscr{D}_\hbar}(\chi_r\overline{\phi_r})-\mathscr{D}_\hbar\phi_\ell\overline{\mathscr{D}_\hbar}(\chi_\ell\chi_r \overline{\phi_r})\right)\dd \sigma\dd\tau\\
&=\int_{S_L} a_\hbar^{-1}\left(\left[-i\hbar\chi'_\ell\phi_\ell+\chi_\ell\mathscr{D}_\hbar\phi_\ell\right]\overline{\mathscr{D}_\hbar}(\chi_r\overline{\phi_r})-\mathscr{D}_\hbar\phi_\ell\left[\chi_\ell\overline{\mathscr{D}_\hbar}(\chi_r\overline{\phi_r})+i\hbar\chi'_\ell\chi_r\overline{\phi_r}\right]\right)\dd\sigma\dd\tau\\
&=-i\hbar\int_{S_L} a_\hbar^{-1}\chi'_\ell\left(\phi_\ell\overline{\mathscr{D}_\hbar}(\chi_r\overline{\phi_r})+\mathscr{D}_\hbar\phi_\ell\left[\chi_r\overline{\phi_r}\right]\right)\dd\sigma\dd\tau\\
&=i\hbar\int_{S_L} a_\hbar^{-1}\chi'_\ell\chi_r\left(\phi_\ell\overline{\mathscr{D}_\hbar\phi_r}+\mathscr{D}_\hbar\phi_\ell\overline{\phi_r}\right)\dd\sigma\dd\tau\\
&=i\hbar\int_{S_L} a_\hbar^{-1}\chi'_\ell\left(\phi_\ell\overline{\mathscr{D}_\hbar\phi_r}+\mathscr{D}_\hbar\phi_\ell\overline{\phi_r}\right)\dd\sigma\dd\tau\,,
\end{split}
\end{equation*}
where we have used $\chi'_\ell\chi'_r=0$ and $\chi'_\ell\chi_r=\chi'_\ell$. Note also that $\chi'_\ell$ is supported in $(-L,0)$. We let $\tilde\phi_\alpha(\sigma, \tau)=e^{i\gamma(\sigma,\tau)/\hbar}\phi_\alpha(\sigma,\tau)$ on $S_L$, where $\gamma$ satisfies $\partial_\sigma\gamma(\sigma,\tau)=\tau-\gamma_0/\hbar-\hbar c_\mu \kappa\frac{\tau^2}{2}$. Using this change of function, we get
\[
w_{\ell,r}=-\int_{S_r} a_\hbar^{-1}\hbar D_\sigma\chi_\ell\left(\tilde\phi_\ell\overline{\hbar D_\sigma\tilde\phi_r}+\hbar D_\sigma\tilde\phi_\ell\overline{\tilde\phi_r}\right)\dd\sigma\dd\tau\,,
\]
where $S_r=(-L,0)\times\R_+$. Then, by integration by parts,
\begin{multline}\label{eq.wtilde}
w_{\ell,r}=\tilde w_{\ell,r}
+i\hbar\int_{0}^{+\infty}a_\hbar^{-1}\left(\tilde\phi_\ell\ \overline{\hbar D_\sigma\tilde\phi_r}+\hbar D_\sigma\tilde\phi_\ell\ \overline{\tilde\phi_r}\right)(0,\tau)\\
-a_\hbar^{-1}\left(\tilde\phi_\ell\ \overline{\hbar D_\sigma\tilde\phi_r}+\hbar D_\sigma\tilde\phi_\ell\ \overline{\tilde\phi_r}\right)(-L,\tau)\dd\tau\,,
\end{multline}
with
\[\tilde w_{\ell,r}=\int_{S_r}\chi_\ell \hbar D_\sigma \left[ a_\hbar^{-1}\left(\tilde\phi_\ell\ \overline{\hbar D_\sigma\tilde\phi_r}+\hbar (D_\sigma\tilde\phi_\ell)\ \overline{\tilde\phi_r}\right)\right]\dd\sigma\dd\tau\,.\]
Note that
\[\tilde w_{\ell,r}=\int_{S_r} \chi_\ell \left(-\tilde\phi_\ell\overline{(\hbar D_\sigma  a_\hbar^{-1}\hbar D_\sigma)\tilde\phi_r}+(\hbar D_\sigma  a_\hbar^{-1}\hbar D_\sigma)\tilde\phi_\ell\ \overline{\tilde\phi_r}\right)\dd\sigma\dd\tau\,,\]
and, coming back to $\phi_\alpha$,
\[\tilde w_{\ell,r}=\int_{S_r} \chi_\ell \left(-\phi_\ell\overline{\mathscr{D}_\hbar a_\hbar^{-1}\mathscr{D}_\hbar\phi_r}+\mathscr{D}_\hbar a_\hbar^{-1}\mathscr{D}_\hbar\phi_\ell\ \overline{\phi_r}\right)\dd\sigma\dd\tau\,.\]
Using the fact that the $\phi_\alpha$ are eigenfunctions associated with the same eigenvalue, we get $\tilde w_{\ell,r}=0$. From \eqref{eq.wtilde}, we deduce that
\begin{multline*}
w_{\ell,r}=
i\hbar\int_{0}^{+\infty}a_\hbar^{-1}\left(\phi_\ell\overline{\mathscr{D}_\hbar\phi_r}+\mathscr{D}_\hbar\phi_\ell\overline{\phi_r}\right)(0,\tau)\\
-a_\hbar^{-1}\left(\phi_\ell\overline{\mathscr{D}_\hbar\phi_r}+\mathscr{D}_\hbar\phi_\ell\overline{\phi_r}\right)(-L,\tau)\dd\tau\,.
\end{multline*}
\end{proof}

\subsubsection{End of the proof of Theorem \ref{thm.tunnel}}\label{sec.fin}
Let us explain how to analyze the asymptotic behavior of the first term, related to the upper part of $\Gamma$, in \eqref{eq.ultime}:
\[w_{\ell,r}^u=\int_{0}^{+\infty}a_\hbar^{-1}\left(\phi_\ell\overline{\mathscr{D}_\hbar\phi_r}+\mathscr{D}_\hbar\phi_\ell\overline{\phi_r}\right)(0,\tau)\dd\tau\,.\]
Note that $a_\hbar=1+o(1)$, $|\hbar c_\mu\tau^2|=o(\hbar^{-2\eta})$. We also recall that $\phi_{\hbar,\ell}$ and $\phi_{\hbar,r}$ are explicitly described in \eqref{eq.phir} and \eqref{eq.phil}. We find that
\begin{multline*}
w_{\ell,r}^u\\
=\int_{0}^{+\infty}a_\hbar^{-1}\left(u_{\hbar,\ell}\overline{(\hbar D_\sigma-\tau+\hbar c_\mu\frac{\kappa}{2}\tau^2)u_{\hbar,r}}+(\hbar D_\sigma-\tau+\hbar c_\mu\frac{\kappa}{2}\tau^2)u_{\hbar,\ell}\overline{u_{\hbar,r}}\right)(0,\tau)\dd\tau\,.
\end{multline*}
Then, we use the uniform approximation given in Proposition \ref{prop:WKB-app}, the explicit expression of the WKB Ansatz in Theorem \ref{thm.BKW}, and the fact that $\Phi_r(0)+\Phi_{\ell}(0)=\mathsf{S}_{\mathsf{u}}$ to get
\begin{multline*}
w_{\ell,r}^u=\\
\int_{0}^{+\infty}a_\hbar^{-1}\left(\Psi_{\hbar,\ell}\overline{\left(\hbar D_\sigma-\tau+\hbar c_\mu\frac{\kappa}{2}\tau^2\right)\Psi_{\hbar,r}}+\left(\hbar D_\sigma-\tau+\hbar c_\mu\frac{\kappa}{2}\tau^2\right)\Psi_{\hbar,\ell}\overline{\Psi_{\hbar,r}}\right)\dd\tau\\
+\mathscr{O}(\hbar^\infty)e^{-\mathsf{S}_{\mathsf{u}}/\hbar^{1/2}}\,,
\end{multline*}
where $\Psi_{\hbar,\ell}(0,\tau)=U\Psi_{\hbar,r}(0,\tau)$ (there is no phase shift since we are at $\sigma=0$). Using again that $\Phi_r(0)+\Phi_{\ell}(0)=S_{\mathsf{u}}$ and the explicit expression of the first term of $\Psi_{\hbar,r}$, we get
\begin{multline*}
\hbar^{\frac14}e^{\mathsf{S}_{\mathsf{u}}/\hbar^{1/2}}w_{\ell,r}^u=\int_{0}^{+\infty}a_\hbar^{-1}Ub_{1,\hbar}\overline{\left(\hbar D_\sigma+\xi_0-\tau+i\hbar^{\frac12}\Phi'_r(0)+\hbar c_\mu\frac{\kappa}{2}\tau^2\right)b_{1,\hbar}}\dd\tau\\
+\int_0^{+\infty}a_\hbar^{-1}\left(\hbar D_\sigma+\xi_0-\tau+i\hbar^{\frac12}\Phi'_{\ell}(0)+\hbar c_\mu\frac{\kappa}{2}\tau^2\right)Ub_{1,\hbar}\overline{b_{1,\hbar}}\dd\tau
+\mathscr{O}(\hbar^\infty)\,.
\end{multline*}
By using that $a_\hbar=1+o(1)$ and the exponential decay of $b_{1,\hbar}$ with respect to $\tau$, we find that
\begin{equation*}
\begin{split}
\hbar^{\frac14}e^{\mathsf{S}_{\mathsf{u}}/\hbar^{1/2}}w_{\ell,r}^u=&\int_{0}^{+\infty}\left(\xi_0-\tau-i\hbar^{\frac12}\Phi'_r(0)\right)Ub_{1,\hbar}\overline{b_{1,\hbar}}(0,\tau)\dd\tau\\
&+\int_0^{+\infty}\left(\xi_0-\tau+i\hbar^{\frac12}\Phi'_{\ell}(0)\right)Ub_{1,\hbar}\overline{b_{1,\hbar}}(0,\tau)\dd\tau
+\mathscr{O}(\hbar)\\
=& 2\int_{0}^{+\infty} (\xi_0-\tau)Ub_{1,\hbar}\overline{b_{1,\hbar}}(0,\tau)\dd\tau-2i\hbar^{\frac12}\Phi'_r(0)\tilde f_{1,0}^2(0)e^{-2i\alpha_{1,0}(0)}\\
&+\mathscr{O}(\hbar)\,,
\end{split}
\end{equation*}
where we used the explicit expression of the first term of $b_{1,\hbar}$ given in \eqref{eq.an0} and Remark \ref{rem.tildef}. Let us now replace $b_{1,\hbar}$ by its first two terms $b_{1,0}+\hbar^{\frac12}b_{1,1}$ (modulo $\mathscr{\hbar}$) given in \eqref{eq.an0}. We recall the following two formulas (see, for instance, \cite[Lemma 5.8]{BHR16}):
\[\int_0^{+\infty}(\xi_0-\tau)u^2_{\xi_0}(\tau)\dd\tau=0\,,\quad 1+2\int_0^{+\infty}(\xi_0-\tau)u_{\xi_0}(\tau) (\partial_{\xi}u)_{\xi_0}(\tau)\dd\tau=\frac{\mu''_1(\xi_0)}{2}\,.\]
With the first formula, we get
\[\begin{split}
2\int_{0}^{+\infty} (\xi_0-\tau)Ub_{1,\hbar}\overline{b_{1,\hbar}}\dd\tau=&2\hbar^{\frac12}\int_{0}^{+\infty}(\xi_0-\tau)\left( Ub_{1,0}\overline{b_{1,1}}+Ub_{1,1}\overline{b_{1,0}}\right)\dd\tau+\mathscr{O}(\hbar)\\
=&-4i\tilde f^2_{1,0}(0)\Phi'_r(0)e^{-2i\alpha_{1,0}(0)}\int_0^{+\infty}(\xi_0-\tau)u_{\xi_0}(\tau)v_{\xi_0}(\tau)\dd\tau\\
&+\mathscr{O}(\hbar)\,.
\end{split}\]
With the second formula, we deduce that
\begin{equation} \label{eq.explicitcomptuwlru}
\hbar^{\frac14}e^{\mathsf{S}_{\mathsf{u}}/\hbar^{1/2}}w_{\ell,r}^u
=-\mu_1''(\xi_0)i\hbar^{\frac12}\Phi'_r(0)\tilde f_{1,0}^2(0)e^{-2i\alpha_{1,0}(0)}+\mathscr{O}(\hbar)\,.
\end{equation}
We recognize here the interaction term associated with the effective Hamiltonian \eqref{eq.effectiveH} (see also \eqref{eq:we}). Using \eqref{eq.normtildef10}, and differentiating at $0$ formula \eqref{eq.defPhi}, the estimate \eqref{eq.explicitcomptuwlru} becomes
\[\hbar^{\frac14}e^{\mathsf{S}_{\mathsf{u}}/\hbar^{1/2}}w_{\ell,r}^u=-i \hbar^{\frac12}\mu_1''(\xi_0) \pi^{-\frac 12} g^{\frac12}\sqrt{V(0)} \mathsf{A}_{\mathsf{u}}e^{-2i\alpha_{1,0}(0)}+\mathscr{O}(\hbar)\,.\]
In the same way, we can deal with the integral corresponding to the down part of the boundary in \eqref{eq.ultime}:
\[w_{\ell,r}^d=-\int_{0}^{+\infty}a_\hbar^{-1}\left(\phi_\ell\overline{\mathscr{D}_\hbar\phi_r}+\mathscr{D}_\hbar\phi_\ell\overline{\phi_r}\right)(-L,\tau)\dd\tau\,,\]
and we find that
\[\hbar^{\frac14}e^{\mathsf{S}_{\mathsf{d}}/\hbar^{1/2}}w_{\ell,r}^d= -i \hbar^{\frac12}\mu_1''(\xi_0) \pi^{-\frac 12} g^{\frac12}\sqrt{V(-L)} \mathsf{A}_{\mathsf{d}}e^{-2i\alpha_{1,0}(-L)}e^{i\left(-2L\gamma_0/\hbar^2+2L\xi_0/\hbar\right)}+\mathscr{O}(\hbar)\,,\]
where there is an additional phase shift coming from the $\pm 2L$-translation in $\sigma$ appearing in the second and first expression in \eqref{eq.phir} and \eqref{eq.phil}, respectively.

Now recalling that, from \eqref{eq.ultime},
\[w_{\ell,r}=i\hbar(w_{\ell,r}^u+w_{\ell,r}^d)\,,\]
that $\hbar=h^{\frac 12}$, and multiplying by $e^{iLf(h)} $ in order to have a more symmetric formula (only the modulus of this quantity is relevant),  we get the expression $\tilde{w}(h)$ given in the statement of Theorem \ref{thm.tunnel}. Then, Theorem \ref{thm.tunnel} follows from \eqref{eq.tunnel0} and Proposition \ref{prop.redNh}.

\bibliographystyle{abbrv}
\bibliography{bibBHR}

\end{document}